\documentclass[12pt]{amsart}

\usepackage{amssymb,color}

\definecolor{MyGreen}{RGB}{29,162,55}   

\numberwithin{equation}{section}

\newtheorem{theorem}[subsection]{Theorem}
\newtheorem{lemma}[subsection]{Lemma}
\newtheorem{Prop}[subsection]{Proposition}

\newtheorem{conjecture}[subsection]{Conjecture}
\newtheorem{remark}[subsection]{Remark}

\theoremstyle{definition}

\newtheorem{Def}[subsection]{Definition}

\theoremstyle{remark}

\newtheorem{example}[subsection]{Example}

\newcommand{\field}{{\bf F}}

\newcommand{\F}{\mathbb{F}}

\newcommand{\GL}{\text{GL}}

\newcommand{\wta}{\widetilde{a}}

\title[Finite subgroups]
{Representations of elementary abelian $p$-groups and finite subgroups of fields}

\author{H.E.A. Campbell}
\address{Department of Mathematics\\
\hfil\break\indent University of New Brunswick, Fredricton, NB, E3B 5A3, Canada }
\email{heac@unb.ca}

\author{J. Chuai}
\address{Department of Mathematics\\
\hfil\break\indent University of New Brunswick, Fredricton, NB, E3B 5A3, Canada }
\email{jchuai@unb.ca}

\author{R.J. Shank}
\address{School of Mathematics, Statistics \&  Actuarial Science \\
\hfil\break\indent University of Kent, Canterbury, CT2 7FS, UK}
\email{R.J.Shank@kent.ac.uk}

\author{D.L. Wehlau}
\address{Department of Mathematics and Computer Science \\
\hfil\break\indent Royal Military College, Kingston, ON, K7K 7B4, Canada}
\email{wehlau@rmc.ca}

\thanks{This research was funded in part by NSERC and the University of New Brunswick.}

\date{\today}

\begin{document}

\begin{abstract}
Suppose $\F$ is a field of prime characteristic $p$ and $E$ is a finite subgroup of the additive group $(\F,+)$.
Then $E$ is an elementary abelian $p$-group. We consider two such subgroups, say $E$ and $E'$, to be equivalent
if there is an $\alpha\in\F^\times:=\F\setminus\{0\}$ such that $E=\alpha E'$. In this paper we show that rational functions can be used to 
distinguish equivalence classes of subgroups and, for subgroups of prime rank or rank less than twelve, we give explicit 
finite sets of  {\it separating invariants}.
\end{abstract}

\maketitle

\section{Introduction}
We are interested in parameterising modular representations of elementary abelian $p$-groups with a view
to computing the associated rings of invariants in a systematic way;
this programme was initiated in \cite{CSW} and continued in \cite{PS}.
We note that, with a few rare exceptions, the modular representation theory of an elementary abelian $p$-group is wild;
see, for example, \cite{benson}.
However, for fixed dimension, it is possible to parameterise the representations.
In Section~2 of \cite{CSW}, the authors compute the rings of invariants for the two dimensional modular
representations of elementary abelian $p$-groups.
Every such  representation is determined, up to group automorphism, by a finite subgroup of the field and 
two such representations are equivalent if and only if the corresponding subgroups are in the same orbit under the
action of the multiplicative group.
More precisely, suppose $\F$ is a field of prime characteristic $p$ and
$E$ is a finite subgroup of the additive group $(\F,+)$.
Then $E$ is an elementary abelian $p$-group. We consider two such subgroups, say $E$ and $E'$, to be equivalent
if there is an $\alpha\in\F^\times:=\F\setminus\{0\}$ such that $E=\alpha E'$.
We will show that rational functions can be used to 
distinguish equivalence classes of subgroups and, for subgroups of prime rank or rank less than twelve, we give explicit 
finite sets of  {\it separating invariants}.

A point $\underline{c}=(c_1,\ldots,c_r)\in\F^r$ determines a finite subgroup
$$E={\rm Span}_{\F_p}\{c_1,\ldots,c_r\}.$$
The rank of $E$ is $r$ if and only if $d_{r,r}(\underline{c})\not=0$, where $d_{r,r}$ is the top Dickson invariant,
see Section ~\ref{prelim} below for details. Two points determine the same subgroup if 
they lie in a common $GL_r(\F_p)$-orbit. Thus equivalence classes of rank $r$ subgroups correspond
to points in the orbit space $GL_r(\F_p)\backslash \F^r\slash\F^\times$ with $d_{r,r}(\underline{c})\not=0$.
We think of this variety as essentially a moduli space for the faithful two dimensional representations of
a rank $r$ elementary abelian $p$-group. However, since we only detect representations up to group automorphism,
it may be more appropriate to think of the variety as parameterising the associated rings of invariants.

In Section~\ref{prelim}, we review motivation and notation. We observe that rational functions in the
Dickson invariants can be used to separate those  orbits with $d_{r,r}\not=0$. A careful analysis
of the rank two case suggests that there are special orbits with
additional structure; Section~\ref{struc_sec} is devoted to an investigation of this structure,
including the assignment of a partition to each orbit.

In Section~\ref{sep_sec}, we assume $\F$ is algebraically closed and construct explicit finite
sets of separating invariants for the orbits for groups of prime rank or of rank less than twelve.
These separating invariants are monomials in $\F[d_{1,r},\ldots,d_{r-1,r},d_{r,r}^{-1}]^{\F^\times}$. 

In Section~\ref{rnk3_sec}, we consider groups of rank three. We give a generating set for the ring of invariants
 $\F[d_{1,3},d_{2,3},d_{3,3}^{-1}]^{\F^\times}$ and use these invariants to determine the partition associated to each orbit.
Section~\ref{rnk4_sec} contains the analogous analysis for groups of rank four.

In Section~\ref{rnk5_sec}, we construct a generating set for
$\F[d_{1,5},d_{2,5},d_{3,5},d_{4,5},d_{5,5}^{-1}]^{\F^\times}$. This involves finding the primitive non-negative integer solutions to
the equation
\begin{equation*}
a_1p^4+a_2p^3\left(\frac{p^2-1}{p-1}\right)
+a_3p^2\left(\frac{p^3-1}{p-1}\right)+a_4p\left(\frac{p^4-1}{p-1}\right)=a_5\left(\frac{p^5-1}{p-1}\right).
\end{equation*}
In Section~\ref{rnk5_par}, we make some progress towards determining the partitions associated to the orbits
of rank five subgroups and conclude with some conjectures on this problem.

\section{Preliminaries}\label{prelim}  

We start with a review of Section 2 of \cite{CSW}. Define a group homomorphism $\rho:(\F,+)\to GL_2(\F)$ by
\[\rho(c):=\left[ \begin{array}{cc} 1&c\\ 0&1 \end{array}\right].\]
The restriction of $\rho$ to a finite subgroup $E$ gives a two dimensional representation of $E$. 
If $E$ and $E'$ are finite subgroups of $(\F,+)$, then $\rho(E)$ and $\rho(E')$ are conjugate if and only if
there exists $\alpha\in \F^\times$ with $E=\alpha E'$.

There is a left action of $GL_2(\F)$ on the polynomial ring $\F[x,y]$ defined by taking $x=[0\, 1]$ and $y=[1\, 0]$.
For a finite subgroup $E\subset \F$, define
\[ \F[x,y]^E:=\{f\in \F[x,y]\mid f\cdot \rho(c)=f,\, \forall c\in E\}.\]
Clearly $x\in\F[x,y]^E$. A second invariant is given by
\[N_E(x,y):=\prod_{c\in E}y\cdot \rho(c)=\prod_{c\in E}(y+cx).\]
If the rank of $E$ is $r$, then using the graded reverse lexicographic order with $y>x$,
the leading term of $N_E(x,y)$ is $y^{p^r}$. Thus $\{x,N_E(x,y)\}$ is a homogeneous system of parameters for
$\F[x,y]^E$ with $\deg(x)\deg(N_E(x,y))=|E|$. Hence applying \cite[3.7.5]{DK} gives
$\F[x,y]^E=\F[x,N_E(x,y)]$.
Define 
\[F_E(t):=\prod_{c\in E}(t-c)\in\F[t].\]
Thus $F_E(t)$ is the monic polynomial in $t$ whose roots are the elements of $E$.
Observe that $F_E(t)=N_E(1,t)$ and $N_E(x,y)=x^{|E|}F_E(y/x)$.
Furthermore, for $\alpha\in\F^\times$, we have $F_{\alpha E}(t)=\alpha^{|E|}F_E(t/\alpha)$.

If $\{x_1,\ldots,x_r\}$ is a basis for $({\F_p}^r)^*$ then
\[D_r(t):=\prod_{u\in({\F_p}^r)^*}(t-u)\in\F_p[x_1,\ldots,x_r]^{GL_r(\F_p)}[t].\]
It is well-known that $D_r(t)=t^{p^r}+\sum_{i=0}^{r-1}d_{r-i,r}t^{p^i}$, where the $d_{j,r}$ are the Dickson invariants, and
\[\F_p[x_1,\ldots,x_r]^{GL_r(\F_p)}=\F_p[d_{1,r},\ldots,d_{r,r}]\]
(see, for example, \cite{Wilkerson}).
Using the inclusion $\F_p\subset \F$, we can interpret elements of 
$\F_p[x_1,\ldots,x_r]$ as polynomial functions on $\F^r$. If $\{c_1,\ldots,c_r\}\subset \F$ is 
an $\F_p$-basis for $E$, then $d_{i,r}(c_1,\ldots,c_r)\in\F$. Since 
$d_{i,r}\in \F_p[x_1,\ldots,x_r]^{GL_r(\F_p)}$, 
the result is independent of the choice of basis for $E$ and so it makes sense to define
$d_{i,r}(E):=d_{i,r}(c_1,\ldots,c_r)$. Since evaluating the coefficients of $D_r(t)$ at a basis for $E$
gives $F_E(t)$, we have the following.

\begin{Prop}
$N_E(x,y)=y^{p^r}+\sum_{i=0}^{r-1}d_{r-i,r}(E)y^{p^i}x^{p^r-p^i}$.
\end{Prop}

Note that, since $d_{r,r}(b_1,\ldots,b_r)$ is the product of the $\F_p$-linear combinations of $\{b_1,\ldots,b_r\}$,
we have ${\rm rank}({\rm Span}_{\F_p}\{b_1,\ldots,b_r\})=r$ if and only if  
$d_{r,r}(b_1,\ldots,b_r)\not=0$.

Our goal is to use rational functions to distinguish equivalence classes of finite subgroups.
A point $(c_1,\ldots,c_r)\in\F^r$ determines a finite subgroup $E={\rm Span}_{\F_p}\{c_1,\ldots,c_r\}$.
The rank of $E$ is $r$ if and only if $d_{r,r}(E)\not=0$. Two points determine the same subgroup if 
they lie in a common $GL_r(\F_p)$-orbit. Thus equivalence classes of rank $r$ subgroups correspond
to points in the orbit space $GL_r(\F_p)\backslash \F^r\slash\F^\times$ with $d_{r,r}(c_1,\ldots,c_r)\not=0$.
Rather than working with the Zariski open set $d_{r,r}\not=0$, it is convenient to introduce a formal inverse 
for $d_{r,r}$, say $t$, and study the resulting Zariski closed set $d_{r,r}t=1$. 
Since $d_{i,r}$ is homogeneous of degree $p^r-p^{r-i}$, for $\alpha\in\F^\times$ we have 
$d_{i,r}(\alpha E)=\alpha^{p^r-p^{r-i}}d_{i,r}(E)$. Thus $\F^\times$ acts on $d_{i,r}$ with weight $p^r-p^{r-i}$.
To ensure that the equation  $d_{r,r}t=1$ is isobaric with weight zero, we assign weight $1-p^r$ to $t$.
This defines an action of $\F^\times$ on $\F[d_{1,r},\ldots,d_{r,r},t]$.

\begin{Prop}\label{sep_prop} Suppose $\F$ is algebraically closed. Then the polynomial functions
$\F[d_{1,r},\ldots,d_{r,r},t]^{\F^\times}$ separate the orbits satisfying $d_{r,r}t=1$.
Furthermore, $\F[d_{1,r},\ldots,d_{r,r},t]^{\F^\times}$ is generated by  $d_{r,r}t$ and a finite  set of monomials
in $\F[d_{1,r},\ldots,d_{r-1,r},t]$. Thus equivalence classes of rank $r$ subgroups
are distinguished by a finite set of rational functions given by monomials in $\F[d_{1,r},\ldots,d_{r-1,r},d_{r,r}^{-1}]$.
\end{Prop}
\begin{proof}
The multiplicative group $\F^\times$ is linearly reductive (see, for example \cite[Theorem~2.2.19]{DK}).
Therefore the ring of invariants is finitely generated and the invariants separate closed orbits. 
For an $\F^\times$-action, if an orbit is not closed then the closure includes one additional point.
This point must be a closed orbit and hence a fixed point. For the given action, the only 
fixed point is the the zero vector. Thus the invariants separate the orbits satisfying $d_{r,r}t=1$.  
Since the action is diagonal, the invariants are generated by a set of monomials.
Note that the weight of $d_{i,r}$ for $i=1,\ldots,r$ 
is positive. Thus every monomial of positive degree and weight zero involves a positive power of $t$. 
\end{proof}

\begin{Def}\label{v1} For a finite subgroup of rank $r$, the following $\F^\times$-invariant rational functions are well-defined:
$$ v_1:=\frac{d_{1,r}^{(p^r-1)/(p-1)}}{d_{r,r}^{p^{r-1}}} \, {\rm \, and \,}\,
v_{1j}:=\frac{d_{j,r}d_{1,r}^{p(p^{r-j}-1)/(p-1)}}{d_{r,r}^{p^{r-j}}}
$$
for $j\in\{2,\ldots,r-1\}$.
\end{Def}

\subsection*{Groups of Rank 2}
Subgroups of rank $2$ provide an instructive special case. 
The group $\F^\times$ acts on $d_{1,2}$ with weight $p^2-p$,
$d_{2,2}$ with weight $p^2-1$ and $t$ with with weight $1-p^2$. 
It is clear that
$$\F[d_{1,2},d_{2,2},t]^{\F^\times}=\F[d_{1,2}^{p+1}t^p,d_{2,2}t].$$
The groups of rank $2$ correspond to orbits with $d_{2,2}t=1$ and the orbits of rank $2$ subgroups are separated
by the rational function $v_1={d_{1,2}^{p+1}}/{d_{2,2}^p}$.
Note that $v_1=v_1(x,y)$ is a rational function of $x,y$ and $v_1(a,b)$ is defined whenever
$\{a,b\}$ is linearly independent over $\F_p$.  
If $V$ is a rank 2 subgroup of $\F$ then we will write $v_1(V) := d_{1,2}^{p+1}(V)/d_{2,2}^p(V)$.
   
In 1911, Dickson \cite{Dickson} showed that $d_{1,2}(x,y) = -\prod_{(a,b)\in \Omega} ax^2 + bxy + y^2$ 
where the product is over all
$a,b \in \F_p$ such that $z^2+bz+a$ is an irreducible quadratic in $\F_p[z]$.  
Thus  $d_{1,2}(1,\gamma)=0$ if and only if $\gamma$ satisfies an irreducible quadratic over $\F_p$.  
In this case, ${\rm Span}_{\F_p}\{1,\gamma\}=\F_{p^2}$.
More generally if $\alpha\neq 0$ then $d_{1,2}(\alpha,\beta) = \alpha^{p^2-p}d_{1,2}(1,\beta/\alpha)=0$ 
if and only if $\beta/\alpha$ satisfies an irreducible quadratic over $\F_p$.  
If $\{\alpha,\beta\}$ is any $\F_p$-basis of $V$ then $\{1,\beta/\alpha\}$ is a basis of $\alpha^{-1}\cdot V$.  
Thus $v_1(V)=0$ if and only if $v_1(\alpha^{-1}\cdot V)=0$ if and only if $d_{1,2}(1,\beta/\alpha)=0$ if and only if
$\alpha^{-1}\cdot V = \F_{p^2}$. In other words, the $\F^\times$-orbit determined by $v_1(V)=0$ consists of the dilations of 
$\F_{p^2}\subset\F$.
This example suggests that some finite subgroups possess additional structure.   

\section{The Structure of a Subgroup}\label{struc_sec}

For a finite subgroup $E\subset (\F,+)$, define the stabiliser 
\[{\rm Stab}_{\F^\times}(E):=\{\alpha\in\F^\times\mid \alpha E= E\}.\]

\begin{Prop}\label{subspace_prop} 
Suppose $E$ a is a finite subgroup of $\F$ of rank $r$.
Then ${\rm Stab}_{\F^\times}(E)=\F_q^\times$ for some finite subfield $\F_q\subseteq \F$ with $q\leq p^r$.
Furthermore, $E$ is a vector space over $\F_{p^s}$ if and only if $\F_{p^s}\subseteq \F_q$.
\end{Prop}
\begin{proof}
Clearly ${\rm Stab}_{\F^\times}(E)$ is closed under multiplication and contains $1$.
Define $L:={\rm Stab}_{\F^\times}(E)\cup\{0\}$. For $\alpha_1,\alpha_2\in L$ and $\gamma\in E$,
\[(\alpha_1+\alpha_2)\gamma=\alpha_1\gamma+\alpha_2\gamma\in E.\] 
Thus $L$ is a subfield of $\F$.
Choose $\gamma\in E\setminus\{0\}$. Then $1\in\gamma^{-1}E$ and 
${\rm Stab}_{\F^\times}(\gamma^{-1}E)={\rm Stab}_{\F^\times}(E)$. Therefore $L\subseteq \gamma^{-1}E$,
$q:=|L|\leq|E|=p^r$ and $E$ is a vector space over $L=\F_q$.
 Furthermore, if $E$ is a
vector space over $\F_{p^s}$, then $\F_{p^s}^\times\subseteq {\rm Stab}_{\F^\times}(E)= \F_q^\times$.
\end{proof}

\begin{theorem}\label{field-thm} Suppose $q=p^s$ and $\F_q\subset \F$. Further suppose that $E$ is a finite subgroup
of $\F$ of rank $r=ms$ for some integer $m$.
Then $E$ is a vector space over $\F_q$ if and only if $d_{i,r}(E)=0$ for
all $i$ not congruent to zero modulo $s$.
\end{theorem}
\begin{proof}
Denote ${\rm hom}_{\F_q}({\F_q}^m,\F_q)$ by $({\F_q}^m)^*$, choose a basis $\{y_1,\ldots,y_m\}$ for
$({\F_q}^m)^*$ and define
\[D_{m,q}(t):=\prod_{u\in({\F_q}^m)^*}(t-u)\in\F_q[y_1,\ldots,y_m]^{GL_m(\F_q)}[t].\]
The coefficients of $D_{m,q}(t)$ are the Dickson invariants over the field $\F_q$. Thus
$D_{m,q}(t)=t^{q^m}+\sum_{j=0}^{m-1} d^{(q)}_{m-j,m}t^{q^j}$.

Suppose $E$ is a vector space over $\F_q$ with basis $\{b_1,\ldots,b_m\}$. Then
evaluating the coefficients of $D_{m,q}(t)$ at $(b_1,\ldots,b_m)$ gives a monic polynomial in
$\F[t]$ whose roots are the elements of $E$, in other words, we recover $F_E(t)$. 
We can also construct $F_E(t)$ by evaluating the coefficients of $D_r(t)$.
By comparing these constructions, we
see that $d_{r-\ell,r}(E)=0$ unless $\ell=js$. Hence $d_{i,r}(E)=0$ unless $i=js-ms=s(j-m)$.

Suppose $d_{i,r}(E)=0$ for all $i$ not congruent to zero modulo $s$.
We will show that $\F_{p^s}\subseteq {\rm Stab}_{\F^\times}(E)$ and, therefore, using 
Proposition~\ref{subspace_prop}, $E$ is a vector space over $\F_{p^s}$.
The polynomial  $d_{i,r}$ is homogeneous of degree $p^r-p^{r-i}=p^{r-i}(p^i-1)$.
If $i=sk$ then $p^s-1$ divides $p^{ks}-1=p^i-1$.
For $\alpha\in \F_{p^s}^\times$, we have $\alpha^{p^s-1}=1$ and $d_{ks,r}(\alpha E)=d_{ks,r}(E)$.
Furthermore, if $d_{i,r}(E)=0$ then $d_{i,r}(\alpha E)=0$.
Thus, if $d_{i,r}(E)=0$ for all $i$ not divisible by $s$, we have $d_{i,r}(\alpha E)=d_{i,r}(E)$
for all $i$ and all $\alpha \in \F_{p^s}^\times$.
The Dickson invariants separate $GL_r(\F_p)$-orbits. Therefore, if
$d_{i,r}(\alpha E)=d_{i,r}(E)$ for all $i$, then $E$ and $\alpha E$ are in the same $GL_r(\F_p)$-orbit,
i.e., if $\{u_1,...,u_r\}$ is a basis for $E$, then $\{\alpha u_1,...,\alpha u_r\}$ is also a
basis for $E$. Thus if $\alpha\in\F_{p^s}^\times$ then $\alpha$ stabilises $E$.
\end{proof}

 Among the subgroups of $(\F,+)$, we distinguish the dilations of finite fields:
    $V_{\alpha,r} := \alpha\cdot\F_{p^r}$ where $\alpha \in \F^\times$ and $r$ is a positive integer.
   Any finite subgroup $V \subseteq (\F,+)$ decomposes as a direct sum
    $$ V =\bigoplus_{j=1}^s V_{\alpha_j,\lambda_j} $$ where $\alpha_j \in \F^\times$ and $\lambda_j \geq 1$.
For example, if $\{\alpha_1,\ldots,\alpha_s\}$ is an $\F_p$-basis for $V$ then  $V=\oplus_{j=1}^s V_{\alpha_j,1}$.
     The decomposition is not necessarily unique.  If $r=m\cdot\ell$, then $\F_{p^r}$ is an $\F_{p^{\ell}}$ vector space
of dimension $m$. If $\{e_1,\ldots,e_m\}$ is an $\F_{p^{\ell}}$-basis for $\F_{p^r}$, then
$V_{\alpha,r}=\oplus_{i=1}^m V_{\alpha e_i,\ell}$.
    For each decomposition we define the sequence $\lambda=(\lambda_1,\lambda_2,\dots,\lambda_s)$.
The sequence is a partition of the integer
     $\dim_{\F_p}(V)$. If we order the sequences lexicographically from the left, then for each $V$ there is a
unique maximal sequence with
$\lambda_1 \geq \lambda_2 \geq \dots \geq \lambda_s$. 
We  will denote this sequence by $\lambda(V)$ and call the sequence the {\it partition}  of $V$, i.e.,  
the partition of $V$ is the largest, with respect to lex order, of the sequences
which correspond to decompositions of $V$. If $V$ has rank $2$, then $\lambda(V)$ is either $(2)$ or $(1,1)$ and
$\lambda(V)=(2)$ if and only if $(d_{1,2}^{p+1}/d_{2,2}^p)(V)=0$.

Define a function $\omega_s:\F\to\F$ by $\omega_s(t)=t^{p^s}-t$. Then $\omega_s$ is $\F_p$-linear and the 
kernel of $\omega_s$ is $\F_{p^s}$. Furthermore, $\omega_s(a)=\omega_s(b)$ if and only if $a-b\in\F_{p^s}$.
Therefore, if $V$ has rank $r$ and  $V\cap\F_{p^s}=\{0\}$ then $\omega_s(V)$ is a finite subgroup of 
rank $r$. In the following, we use the convention that $d_{0,r}=1$ and that $d_{i,r}(V)=0$ if $i<0$ or 
$i>{\rm rank}(V)$.

\begin{theorem}\label{comp-thm}
If $V$ is a finite subgroup of rank $r$ and $V=\field_{p^s}\oplus W$, where
$W$ is a finite subgroup of rank $\ell=r-s$,
then  $$d_{i,r}(V)=d_{i,\ell}\left(\omega_s\left(W\right)\right)-d_{i-s,\ell}\left(\omega_s\left(W\right)\right).$$
\end{theorem}

\begin{proof}
Recall that 
$$N_V(x,y)=\sum_{i=0}^r d_{r-i,r}(V)y^{p^i}x^{p^r-p^i}.$$
Furthermore, since the elements of $\F_{p^s}$ are the roots of $t^{p^s}-t$,
we have $N_{\F_{p^s}}(x,y)=y^{p^s}-y x^{p^s-1}$. Thus $d_{s,s}(\F_{p^s})=-1$ and
$d_{i,s}(\F_{p^s})=0$ for $0<i<s$.
Using the decomposition of $V$, we get
\begin{eqnarray*}
N_V(x,y)&=&\prod_{v\in V}(y+vx)=\prod_{a\in\F_{p^s},b\in W}(y+(a+b)x)\\
&=&\prod_{b\in W}\left(\prod_{a\in\F_{p^s}}\left(\left(y+bx\right)+ax\right)\right)\\
&=&\prod_{b\in W} N_{\F_{p^s}}(x,y+bx)=\prod_{b\in W}\left((y+bx)^{p^s}-(y+bx)x^{p^s-1}\right)\\
&=&\left(x^{p^s-1}\right)^{p^{r-s}}\prod_{b\in W}\left(\left(\frac{y^{p^s}}{x^{p^s-1}}-y\right)+(b^{p^s}-b)x\right)\\
&=& \left(x^{p^s-1}\right)^{p^{r-s}}N_{\omega_s(W)}\left(x,\frac{y^{p^s}}{x^{p^s-1}}-y\right).
\end{eqnarray*}

Since $N_E(x,y)$ is an $\F_p$-linear function of $y$,
\begin{eqnarray*}
N_V(x,y)&=& \left(x^{p^s-1}\right)^{p^{r-s}}\left(N_{\omega_s(W)}\left(x,\frac{y^{p^s}}{x^{p^s-1}}\right)-N_{\omega_s(W)}(x,y)\right)\\
&=&\left(\sum_{i=0}^{r-s}d_{r-s-i,\ell}\left(\omega_s\left(W\right)\right)y^{p^{s+i}}x^{p^r-p^{s+i}}\right)
-\left(\sum_{i=0}^{r-s}d_{r-s-i,\ell}\left(\omega_s\left(W\right)\right)y^{p^i}x^{p^r-p^i}\right)\\
&=&\left(\sum_{j=s}^{r}d_{r-j,\ell}\left(\omega_s\left(W\right)\right)y^{p^{j}}x^{p^r-p^{j}}\right)
-\left(\sum_{i=0}^{r-s}d_{r-s-i,\ell}\left(\omega_s\left(W\right)\right)y^{p^i}x^{p^r-p^i}\right)\\
&=&\sum_{k=0}^{r}\left(d_{r-k,\ell}\left(\omega_s\left(W\right)\right)
-d_{r-s-k,\ell}\left(\omega_s\left(W\right)\right)\right)y^{p^{k}}x^{p^r-p^k}.
\end{eqnarray*}
Thus $d_{i,r}(V)=d_{i,\ell}\left(\omega_s\left(W\right)\right)-d_{i-s,\ell}\left(\omega_s\left(W\right)\right)$.
\end{proof}

\begin{theorem}\label{codim1} Suppose $V$ is a finite subgroup of rank $s+1$ with $s>1$. 
Then the partition of $V$ is $(s,1)$ if and only if  
$d_{i,s+1}(V)=0$ for $1<i<s$ and
\[\left(\frac{d_{1,s+1}^pd_{s,s+1}}{d_{s+1,s+1}^p}\right)\left(V\right)=1.\]
\end{theorem}
\begin{proof} Suppose $V$ has partition $(s,1)$. 
Then $\alpha V=\F_{p^s}\oplus \beta F_p$ for some $\alpha\in\F^\times$ and 
$\beta\in\F\setminus\F_{p^s}$. 
As in Section~\ref{prelim}, denote
$v_{1s}=(d_{1,s+1}^pd_{s,s+1})/(d_{s+1,s+1}^p)$. Since $v_{1s}$ is $\F^\times$-invariant, 
$v_{1s}(V)=v_{1s}(\F_{p^s}\oplus \beta F_p)$. Furthermore,
$d_{i,s+1}(V)=0$ if and only if $d_{i,s+1}(\alpha V)=0$.
Therefore, using Theorem~\ref{comp-thm},
 $d_{i,s+1}(V)=0$ for $1<i<s$, 
$d_{1,s+1}(V)=-d_{s+1,s+1}(V)$ and $d_{s,s+1}(V)=-1$, giving 
$v_{1s}(V)=1$ as required.

Define $f:=d_{s+1,s+1}^p-d_{1,s+1}^pd_{s,s+1}$. We will show that if $f(V)=0$ and
$d_{i,s+1}(V)=0$ for $1<i<s$, then $V$ contains a subgroup equivalent to
$\F_{p^s}$.

Suppose $V$ has basis $\{c_1,\ldots,c_{s+1}\}$. Using Theorem~\ref{field-thm},
if $d_{i,s}(c_1,\ldots,c_s)=0$ for $0<i<s$, then ${\rm Span}_{\F_p}\{c_1,\ldots,c_s\}$ is
 equivalent to $\F_{p^s}$. Thus, to show that the partition of $V$ is $(s,1)$, it is sufficient
to find a $g\in GL_{s+1}(\F_p)$ such that $d_{i,s}g(c_1,\ldots,c_{s+1})=0$ for $0<i<s$.
(Since $d_{i,s}\in\F_p[x_1,\ldots,x_s]\subset \F[x_1,\ldots,x_{s+1}]$, we can evaluate $d_{i,s}$ on an
element of $\F^{s+1}$.)

Recall that $d_{i,s+1}=d_{i,s}^p-d_{i-1,s}D_s^{p-1}(x_{s+1})$ (see Proposition 1.3(b) of \cite{Wilkerson}).
In particular $d_{1,s+1}=d_{1,s}^p-D_s^{p-1}(x_{s+1})$ and $d_{s+1,s+1}=-d_{s,s}D_s^{p-1}(x_{s+1})$.
Suppose there is a basis $\{b_1,\ldots,b_{s+1}\}$ for $V$ with $d_{1,s}(b_1,\ldots,b_{s+1})=0$.
Then, since $d_{2,s}^p=d_{2,s+1}+d_{1,s}D_s^{p-1}(x_{s+1})$, we have $d_{2,s}(b_1,\ldots,b_{s+1})=0$.
Similarly, $d_{i,s}(b_1,\ldots,b_{s+1})=0$ for $i=3,\ldots,s-1$. Thus it is sufficient to find a basis for
which $d_{1,s}(b_1,\ldots,b_{s+1})=0$.

Let $\Pi_s$ denote the $GL_{s+1}(\F_p)$-orbit product of $d_{1,s}$. The degree of $\Pi_s$ is
$\deg(d_{1,s})(p^{s+1}-1)/(p-1)=p^{s-1}(p^{s+1}-1)$ (look at the stabiliser of the hyperplane $x_{s+1}=0$). 
Define 
\[H:=f^{p^{s-2}}+\sum_{j=1}^{s-2}(-1)^{j+1}d_{s-j,s+1}^{p^{s-2-j}}d_{1,s+1}^{p^{(s-1-j)}(p^{j+1}-1)/(p-1)}.\]
Note that $H$ is non-zero and homogeneous of degree $p^{s-1}(p^{s+1}-1)$.
We will show that $d_{1,s}$ divides $H$ and, therefore, $H$ is non-zero scalar multiple of $\Pi_s$.
Since $H(V)=0$, we then conclude that $\Pi_s(V)=0$ and, therefore,  $d_{1,s}(gV)=0$ for some $g$.
Hence the partition of $V$ is $(s,1)$.

We work modulo the ideal $\mathfrak n:=\langle d_{1,s}\rangle$. 
We will use induction on $i$ to prove that
$$f^{p^i}+\sum_{j=1}^{i}(-1)^{j+1}d_{s-j,s+1}^{p^{i-j}}d_{1,s+1}^{p^{(i-j+1)}(p^{j+1}-1)/(p-1)}\equiv_{\mathfrak n}
(-1)^{i+1}d_{s-1-i,s}d_{1,s+1}^{(p^{i+2}-1)/(p-1)}$$
for $i\in\{0,\ldots,s-2\}$. For $i=s-2$, the left side of this equivalence is $H$ and the right is congruent to zero. 

For convenience, we use $h$ to denote $D_s^{p-1}(x_{s+1})$. Thus 
$$d_{i,s+1}=d_{i,s}^p-d_{i-1,s}h\in \F_p[d_{1,s},\ldots,d_{i,s},h]$$
and $d_{1,s+1}\equiv_{\mathfrak n}-h$. Using this equivalence,
$d_{s+1,s+1}=-d_{s,s}h\equiv_{\mathfrak n} d_{s,s}d_{1,s+1}$ and
$d_{s,s+1}\equiv_{\mathfrak n}d_{s,s}^p+d_{s-1,s}d_{1,s+1}$. Thus
$$f=d_{s+1,s+1}^p-d_{1,s+1}^pd_{s,s+1}\equiv_{\mathfrak n} d_{s,s}^pd_{1,s+1}^p-d_{1,s+1}^p\left(d_{s,s}^p+d_{s-1,s}d_{1,s+1}\right)
=-d_{s-1,s}d_{1,s+1}^{p+1}$$
which is the $i=0$ case.

Assume the result is true for $i$ and take the $p^{th}$ power:
$$
f^{p^{i+1}}+\sum_{j=1}^{i}(-1)^{j+1}d_{s-j,s+1}^{p^{i-j+1}}d_{1,s+1}^{p^{(i-j+2)}(p^{j+1}-1)/(p-1)}\equiv_{\mathfrak n}
(-1)^{i+1}d^p_{s-1-i,s}d_{1,s+1}^{p(p^{i+2}-1)/(p-1)}.$$
However, $d_{s-i-1,s}^p\equiv_{\mathfrak n}d_{s-1-i,s+1}-d_{1,s+1}d_{s-i-2,s}$. Substituting this into the 
right side of the equivalence and bringing one of the resulting terms to the left gives
\begin{eqnarray*}
f^{p^{i+1}}&+&\sum_{j=1}^{i+1}(-1)^{j+1}d_{s-j,s+1}^{p^{i+1-j}}d_{1,s+1}^{p^{((i+1)-j+1)}(p^{j+1}-1)/(p-1)}\\
&\equiv_{\mathfrak n}&
(-1)^{(i+1)+1}d_{s-1-(i+1),s}d_{1,s+1}^{(p^{(i+1)+2}-1)/(p-1)},
\end{eqnarray*}
as required.
\end{proof}

\begin{theorem}\label{embedding-thm} Suppose $\F$ is algebraically closed and 
$V$ is a finite subgroup of rank $r=s-1\geq 2$. 
Then $V$ is a subgroup of a dilation of $\F_{p^s}$ if and only if $d_{1,r}(V)\not=0$, $v^p_{12}(V)=v_1(V)$ and 
$v_{1i}^p(V)=v_{1(i-1)}(V)$
for $i\in\{3,\ldots,r-1\}$.
\end{theorem}

\begin{proof}
As in the proof of Theorem~\ref{codim1}, we use the equations 
$$d_{i,s}=d_{i,r}^p-d_{i-1,r}D_{r}^{p-1}(x_s).$$ Choose a basis $\{c_1,\ldots,c_r\}$ for $V$.
Note that $D_r(t)\in \F[x_1,\ldots,x_r][t]$ and let $\widetilde{D_r}(t)$ denote the polynomial in $\F[t]$
constructed by setting $x_i$ to $c_i$ in $D_r(t)$.
Since $\F$ is algebraically closed, $\F$ contains a root of 
$\widetilde{D_r}(t)^{p-1}-d^p_{1,r}(V)$, say
$c_s$. Take $E$ to be the $\F_p$-span of $\{c_1,\ldots,c_s\}$. 
Thus $\widetilde{D_r}(c_s)^{p-1}=d_{1,r}^p(V)$ and
$d_{i,s}(E)=\left(d_{i,r}^p-d_{i-1,r}d_{1,r}^p\right)(V)$.
In particular, $d_{s,s}(E)=-d_{r,r}(V)d_{1,r}(V)$.
Hence $E$ has rank $s$ if and only if $d_{1,r}(V)\not=0$.

Using Theorem~\ref{field-thm}, $E$ is a dilation of $\F_{p^s}$
if and only if $d_{1,r}(V)\not=0$
and $0=d_{i,s}(E)= \left(d_{i,r}^p-d_{i-1,r}d_{1,r}^p\right)(V)$
for $i\in\{1,\ldots,s-1\}$. For $i>2$, if we multiply 
$d_{i,r}^p-d_{i-1,r}d_{1,r}^p$ by $d_{1,r}^{p^2(p^{r-i}-1)/(p-1)}/d_{r,r}^{p-i+1}$,
we get the $\F^\times$-invariant form $v_{1i}^p-v_{1(i-1)}$. For $i=2$,
the invariant form is $v_{12}^p-v_1$.
Thus, if $d_{1,r}(V)\not=0$, $\left(v_{12}^p-v_1\right)(V)=0$ and $\left(v_{1i}^p-v_{1(i-1)}\right)(V)=0$ for
$i\in\{3,\ldots,r-1\}$,
$V$ is contained in a dilation of $\F_{p^s}$. To prove the converse, observe that
if $V\subset W$ and $W$ is a dilation of $\F_{p^s}$ then the equations are easily seen to be satisfied.
\end{proof}

\section{Separating Invariants}\label{sep_sec}

It follows from Proposition~\ref{sep_prop} that, for $\F$ algebraically closed, the orbits of finite subgroups of rank $r$ are separated by 
the rational functions 
$$\F[d_{1,r},\ldots,d_{r-1,r},d_{r,r}^{-1}]^{\F^\times}.$$ 
In the following, whenever we consider separating invariants, we assume $\F$ is algebraically closed.
The ring $\F[d_{1,r},\ldots,d_{r-1,r},d_{r,r}^{-1}]^{\F^\times}$ is spanned by the monomials
$d_{1,r}^{a_1}\cdots d_{r-1,r}^{a_{r-1}}d_{r,r}^{-a_r}$ satisfying the equation
\begin{equation}\label{inv_eq}
        a_1p^{r-1}(p-1)+a_2p^{r-2}(p^2-1)+\cdots + a_{r-1}p(p^{r-1}-1)=a_r(p^r-1).
\end{equation}
The non-negative integer solutions to Equation~\ref{inv_eq} form a monoid. This monoid inherits a partial order from the integer lattice
${\mathbb Z}^r$. We will refer to a non-zero element of the monoid as 
{\it primitive} if it can not be written as a sum of two smaller elements. 
We will refer to a monomial with a primitive exponent sequence as a  primitive invariant.
It is clear that the set of primitive elements in the monoid 
give the exponents of a generating set for the ring of invariants. We will often abuse notation by using the same symbol to denote
the sequence $(a_1,\ldots,a_r)$ and the monomial  $d_{1,r}^{a_1}\cdots d_{r-1,r}^{a_{r-1}}d_{r,r}^{-a_r}$. 
For every $i<r$, let $\ell$ denote the greatest common divisor of $p^i-1$ and $p^r-1$. Then
$$v_i:=d_{i,r}^ad_{r,r}^{-b}\, ,$$ with $a=(p^r-1)/\ell$ and $b=p^{r-i}(p^i-1)/\ell$, is the unique primitive invariant with support 
$\{d_{i,r}, d_{r,r}^{-1}\}$. Note that this definition of $v_i$ is consistent with the definition of $v_1$ given in Definition~\ref{v1}.

If $(a_1,\ldots,a_r)$ is solution to Equation~\ref{inv_eq}, we will refer to $a_1+a_2+\cdots +a_{r-1}-a_{r}$ as the
{\it height} of the solution. This is just the degree of the monomial $d_{1,r}^{a_1}\cdots d_{r-1,r}^{a_{r-1}}d_{r,r}^{-a_r}$ if
we assign degree one to each $d_{i,r}$. Thus a primitive solution has height at least one and any non-negative solution of 
height one is primitive.


\begin{lemma}\label{vijlemma}
If the greatest common divisor of $r$ and $i$ divides $j$, then 
there exists a primitive invariant
$$v_{ij}:=d_{j,r}d_{i,r}^ad_{r,r}^{-b}.$$
\end{lemma}
\begin{proof} To find invariants with support $\{d_{i,r},d_{j,r},d_{r,r}^{-1}\}$, we consider the equation
$$a_jp^{r-j}(p^j-1)+a_ip^{r-i}(p^i-1)=a_r(p^r-1).$$
Suppose $m$ is the greatest common divisor of $i$ and $r$. Then $p^m-1$ is the greatest common divisor of
$p^r-1$ and $p^{r-i}(p^i-1)$. 
Thus there exist positive integers
$\bar{a}$ and $\bar{b}$ with $p^m-1=\bar{b}(p^r-1)-\bar{a}p^{r-i}(p^i-1)$. 

Since $m$ divides $j$, $p^m-1$ divides $p^j-1$.
Thus $c:=p^{r-j}(p^j-1)/(p^m-1)$ is a positive integer. Define 
$a:=c\bar{a}$ and $b:=c\bar{b}$. Then $p^{r-j}(p^j-1)=b(p^r-1)-ap^{r-i}(p^i-1)$. 
Hence $a_j=1$, $a_i=a$ and $a_r=b$
is a solution to the above equation. Choose the minimal positive $a$ and $b$ to get the primitive solution.
\end{proof}

\begin{theorem}\label{prime_sep}
If $r$ is prime, then the orbits of finite subgroups of rank $r$ are separated by
$$\{v_i\mid i=1,\ldots,r-1\}\cup \{v_{ij}\mid 1\le i<j<r\}.$$
\end{theorem}
\begin{proof}
Consider an arbitrary non-negative non-zero solution  $(a_1,\ldots,a_r)$ to Equation~\ref{inv_eq} and define
$f:=d_{1,r}^{a_1}\cdots d_{r-1,r}^{a_{r-1}}d_{r,r}^{-a_r}$. Suppose $E$ is a rank $r$ subgroup. If $f(E)=0$,
then $d_{i,r}(E)=0$ for some $i<r$ and $f(E)$ is determined by $v_i(E)$. Thus we may assume 
$f(E)\not=0$. 

 Since $r$ is prime, it follows from Lemma~\ref{vijlemma}, that the invariants $v_{ij}$ exist. Thus we can
write
$$f =v_{12}^{a_2}v_{13}^{a_3}\cdots v_{1(r-1)}^{a_{r-1}}d_{1,r}^ad_{r,r}^{-b}$$
for some, not necessarily positive, integers $a$ and $b$. 
Since
$f/(v_{12}^{a_2}v_{13}^{a_3}\cdots v_{1(r-1)}^{a_{r-1}})$ is invariant,
$d_{1,r}^ad_{r,r}^{-b}=v_1^c$ for some integer $c$. 
If $d_{1,r}(E)\not=0$, then $v_1^c(E)$ is well-defined even if $c$ is negative and  
$f(E)$ is determined by $v_1(E)$ and $v_{1j}(E)$ for $j>1$.

Suppose $d_{1,r}(E)=0$. Then, since $f(E)\not=0$,  we have $a_1=0$. Hence
$$f=v_{23}^{a_3}v_{24}^{a_4}\cdots v_{2(r-1)}^{a_{r-1}}v_2^a$$
for some,  not necessarily positive, integer $a$. If $d_{2,r}(E)\not=0$, then $v_2^a(E)$ is well-defined and $f(E)$ is determined by
$v_2(E)$ and $v_{2j}(E)$ for $j>2$.

Continuing in this fashion, the problem is reduced to $d_{i,r}(E)=0$ for $i<r-1$ and $f(E)\not=0$.
Thus $f$ has support $\{d_{r,r-1},d_{r,r}^{-1}\}$ and is a positive power of $v_{r-1}$.
\end{proof}

\begin{theorem}\label{rk4_sep} The orbits of finite subgroups of rank $4$ are separated by
$\{v_1,v_2,v_3,v_{12}, v_{13},v_{32}\}$.
\end{theorem}
\begin{proof} The proof is similar to the proof of Theorem~\ref{prime_sep}. Using Lemma~\ref{vijlemma},
the invariants $v_{12}$, $v_{13}$ and $v_{32}$ exist. The set $\{v_1, v_{12},v_{13}\}$ separates orbits with
$d_{1,4}(E)\not=0$, the set $\{v_3,v_{32}\}$ separates orbits with $d_{1,4}(E)=0$, $d_{3,4}(E)\not =0$, and 
$v_2$ separates orbits with $d_{1,4}(E)=0$, $d_{3,4}(E)=0$ and $d_{2,4}(E)\not=0$.
\end{proof}

\begin{lemma}\label{u2jlemma}
Suppose the greatest common divisor of $r$ and $i$ is $2$ and $j$ is odd.
If  $d_{j,r}^{a_j}d_{i,r}^{a_i}d_{r,r}^{-a_r}$ is a primitive invariant, then
$p+1$ divides $a_j$. Furthermore, there exists a primitive invariant
$$u_{ij}:=d_{j,r}^{p+1}d_{i,r}^ad_{r,r}^{-b}.$$
\end{lemma}
\begin{proof}
To find invariants with support $\{d_{i,r},d_{j,r},d_{r,r}^{-1}\}$, we consider the equation
$$a_jp^{r-j}(p^j-1)=a_r(p^r-1)-a_ip^{r-i}(p^i-1).$$
Since $r$ and $i$ are even, $p^2-1$ divides both $p^r-1$ and $p^i-1$. Hence $p^2-1$ divides
$a_jp^{r-j}(p^j-1)$. 
We will show that $p+1$ divides $a_j$. Since $j$ is odd, dividing $p^j-1$ by $p+1$ gives remainder $p-1$.
Therefore $\gcd(p^j-1,p+1)=\gcd(p+1,p-1)$. Hence $\gcd((p^j-1)/(p-1),p+1)=1$.
Thus $p+1$ divides $a_j$.

Since $2$ is the greatest common divisor of $r$ and $i$,
there exist positive integers
$\bar{a}$ and $\bar{b}$ with $p^2-1=\bar{b}(p^r-1)-\bar{a}p^{r-i}(p^i-1)$. 
Note that $c:=p^{r-j}(p^j-1)/(p-1)$ is a positive integer. Define 
$a:=c\bar{a}$ and $b:=c\bar{b}$. Then $(p+1)p^{r-j}(p^j-1)=b(p^r-1)-ap^{r-i}(p^i-1)$. 
Hence $a_j=p+1$, $a_i=a$ and $a_r=b$
is a solution to the above equation. Choose the minimal positive $a$ and $b$ to get the primitive solution.
\end{proof}

\begin{theorem}
The orbits of finite subgroups of rank $6$ are separated by
$\{v_1,v_2,v_3,v_4,v_5,v_{12},v_{13},v_{14},v_{15},v_{52},v_{53},v_{54},v_{24},u_{23},u_{43}\}$.
\end{theorem}
\begin{proof} The proof is similar to the proof of Theorem~\ref{prime_sep}. Using Lemmas~\ref{vijlemma} and \ref{u2jlemma}, 
the listed invariants exist. The set $\{v_1,v_{12},v_{13},v_{14},v_{15}\}$ separates orbits with $d_{1,6}(E)\not=0$.
The set $\{v_5,v_{52},v_{53},v_{54}\}$ separates orbits with $d_{1,6}(E)=0$ and $d_{5,6}(E)\not=0$.
The set $\{v_2,v_{24},u_{23}\}$ separates orbits with $d_{1,6}(E)=0$, $d_{5,6}(E)=0$ and $d_{2,6}(E)\not=0$.
The set $\{v_4,u_{43}\}$ separates orbits with $d_{1,6}(E)=0$, $d_{5,6}(E)=0$, $d_{2,6}(E)=0$ and $d_{4,6}(E)\not=0$.
Finally, $v_3$ separates orbits with $d_{1,6}(E)=0$, $d_{5,6}(E)=0$, $d_{2,6}(E)=0$, $d_{4,6}(E)=0$ and $d_{3,6}(E)\not=0$
\end{proof}

\begin{theorem}
The orbits of finite subgroups of rank $8$ are separated by
$$\{v_i\mid 1\le i< 8\} \cup\{v_{ij}\mid i\, {\rm odd}\, , i<j\}\cup \{v_{24},v_{26},v_{64}\}.$$
\end{theorem}
\begin{proof} The proof is similar to the proof of Theorem~\ref{prime_sep}. Using Lemma~\ref{vijlemma}, 
the listed invariants exist. 
We then ``eliminate'' variables in the following order: $1,3,5,7,2,6,4$. 
\end{proof}

\begin{theorem}
The orbits of finite subgroups of rank $9$ are separated by
$$\{v_i\mid 1\le i< 9\} \cup\{v_{ij}\mid i\not\in\{3,6\}\, , i<j\}\cup \{v_{36}\}.$$
\end{theorem}
\begin{proof} The proof is similar to the proof of Theorem~\ref{prime_sep}. Using Lemma~\ref{vijlemma}, 
the listed invariants exist. 
We then ``eliminate'' variables in the following order: $1,2,4,5,7,8,3,6$. 
\end{proof}

\begin{theorem}
The orbits of finite subgroups of rank $10$ are separated by
$$\{v_i\mid 1\le i< 10\} \cup\{v_{ij}\mid i\in\{1,3,7,9\}\, , i<j\}\cup \{v_{ij},u_{i5}\mid i,j \, {\rm even},\, i<j\}.$$
\end{theorem}
\begin{proof} The proof is similar to the proof of Theorem~\ref{prime_sep}. Using Lemmas~\ref{vijlemma} and \ref{u2jlemma}, 
the listed invariants exist. 
We then ``eliminate'' variables in the following order: $1,3,7,9,2,4,6,8,5$. 
\end{proof}

  \section{Groups of Rank 3}\label{rnk3_sec}

In this section we consider finite subgroups of rank $3$. It follows from Theorem~\ref{prime_sep}
that the orbits of rank $3$ subgroups are separated by $v_1$, $v_2$ and $v_{12}$. 
To construct a generating set for the invariants we add the intersection of the integer lattice with 
the line segment joining $v_2$ and $v_{12}$ 

    \begin{theorem}
      $\F[d_{1,3},d_{2,3},d_{3,3}^{-1}]^{\F^\times}$ is generated by $v_1$, $v_{12}$, $v_2$ and
      $$f_i = d_{1,3}^{p-i} d_{2,3}^{i(p+1)+1} {d_{3,3}}^{-(i+1)p}$$ for $i=1,2,\ldots,p-1$.
    \end{theorem}
    \begin{proof}
    In this case Equation~\ref{inv_eq} becomes $a_1p^2+a_2(p^2+p)=a_3(p^2+p+1)$. 
We will work with exponent sequences so $v_1=(p^2+p+1,0,p^2)$, $v_{12}=(p,1,p)$,
$v_2=(0,p^2+p+1,p^2+p)$ and
$$f_i=v_{12}+\frac{i}{p}(v_2-v_{12})=(p,1,p)+i(-1,p+1,p)=(p-i,(i+1)p+1,(i+1)p).$$
From this it is not hard to see that the primitive solutions of height one are precisely 
$v_2$, $v_{12}$ and $f_i$ for $i=1,\ldots, p-1$. Note that $v_1$ has height $p+1$.
We will show that there are no additional primitive solutions. Suppose $Q=(a_1,a_2,a_3)$
is an additional primitive solution. 
If $a_1=0$, then $Q$ is a multiple of $v_2$, and if $a_2=0$, then $Q$ is a multiple $v_1$.
Thus we may assume $a_1\ge 1$ and $a_2\ge 1$. If $a_1\ge p$, then $Q-v_{12}$ is a non-negative solution, 
so we may assume $p> a_1 \ge 1$. We may also assume that the  height of $Q$ is at least two.
    
Since $p$ divides $a_3$, we can write $a_3=p\wta_3$ with $\wta_3$ a positive integer. 
Hence the equation becomes  $$a_1p+a_2(p+1)=\wta_3(p^2+p+1).$$
Write $a_1=p-i$ for $i\in \{1,\ldots,p-1\}$. We will show that $Q-f_i$ is non-negative, contradicting the assumption that
$Q$ was an additional primitive solution.
Observe that $\wta_3\equiv i+1 \pmod{p+1}$. Since $i+1\in\{2,\ldots,p\}$ and $\wta_3$ is positive,
this means $\wta_3\ge i+1$ and $a_3\ge p(i+1)$. Since the height is at least two,
$$a_2\ge 2+a_3-a_1=a_3+i-p+2\ge p(i+1)+i-p+2=(p+1)i+2.$$
Hence $Q-f_i$ is non-negative, as required.
     \end{proof}

     If $V$ is a subgroup of rank 3, then $\lambda(V)=(3)$, or $\lambda(V)=(2,1)$ or $\lambda(V)=(1,1,1)$. 
Applying Theorems~\ref{field-thm} and \ref{codim1} gives the following.
     \begin{theorem}
       Let $V$ be a rank 3 subgroup of $(\F,+)$.  Then
       \begin{enumerate}
          \item[(i)] $\lambda(V)=(2,1)$ if and only if $v_{12}(V)=1$. \label{caseI}
           \item[(ii)] $\lambda(V)=(3)$ if and only if 
                          $d_{1,3}(V)=d_{2,3}(V)=0$ and $d_{3,3}(V) \neq 0$. \label{caseII}
       \end{enumerate}
     \end{theorem}

\section{Groups of Rank 4}\label{rnk4_sec}

In this section we consider finite subgroups of rank $4$. It follows from Theorem~\ref{rk4_sep}
that the orbits of rank $4$ subgroups are separated by $v_1$, $v_2$, $v_3$, $v_{12}$, $v_{13}$ and $v_{32}$.
To extend this to a generating set, we need to identify the primitive solutions to the equation
\begin{equation}\label{rk4eq}
        a_1(p^4-p^3)+a_2(p^4-p^2)+a_3(p^4-p)=a_4(p^4-1).
\end{equation}
Using Lemma~\ref{u2jlemma}, one additional solution is given by $u_{23}$.
It will be convenient to work with the exponent sequences:
$$v_1=\left(\frac{p^4-1}{p-1},0,0,p^3\right),\, v_2=\left(0,p^2+1,0,p^2\right),\, 
v_3=\left(0,0,\frac{p^4-1}{p-1},\frac{p(p^3-1)}{p-1}\right),$$
$$v_{12}=(p^2+p,1,0,p^2), \, v_{13}:= (p,0,1,p), \, u_{23}=(0,p^2,p+1,p^2+p).$$


To generate the ring of invariants, we need the following additional solutions:
\begin{itemize}
\item[(i)] lattice points on the line segment joining $v_2$ to $v_{12}$:
$$\mathcal L :=\{((p+1)j, p^2+1-jp,0,p^2)\mid 0\le j\le p\};$$
\item[(ii)] lattice points in the triangle with vertices $v_{13}$, $v_3$ and $u_{23}$:
$$
\Delta:=\{iv_{13}/p+ju_{23}/p^2+(p^2-pi-j)v_3/p^2\mid 0\le i\leq p,0\le j\le p^2-ip \}.$$
\end{itemize}

Note that
\begin{eqnarray*}
iv_{13}/p&+&ju_{23}/p^2+(p^2-pi-j)v_3/p^2\\&=&(i,j,p^3+(1-i)p^2+(1-i-j)(p+1), p^3+(1-i)p^2+(1-i-j)p)
\end{eqnarray*}
and $v_{32}$ is given by taking $i=0$ and $j=1$. 

\begin{theorem}\label{rk4gen} $\F[\F^5]^{\GL_4(\F_p)\times \F^\times}$ is minimally generated as an $\F$-algebra by $d_{4,4}t$ and
$$\{d_{1,4}^{a_1} d_{2,4}^{a_2} d_{3,4}^{a_3}t^{a_4}\mid (a_1,a_2,a_3,a_4)\in \{v_1\}\cup\mathcal L\cup\Delta\}.$$
\end{theorem}

\begin{proof}
We prove the theorem by identifying the primitive solutions to Equation~\ref{rk4eq}.
Suppose, by way of contradiction $Q=(a_1,a_2,a_3,a_4)$ is a primitive solution not included in the above list.

We start by considering solutions with $a_3=0$. In this case $p^2$ divides $a_4$ and $p+1$ divides $a_1$.
Writing $a_4=\bar{a}_4 p^2$ and $a_1=\bar{a}_1(p+1)$, Equation~\ref{rk4eq} becomes
$\bar{a}_1p+a_2=\bar{a}_4(p^2+1)$. If $a_2=0$, then $Q$ is a multiple of $v_1$. If $a_2>0$ and $\bar{a}_1\ge p$,
then $Q-v_{12}$ is non-negative. Thus we may assume $\bar{a}_1<p$. It is then easy to see that 
$Q-u$ is non-negative for some $u\in\mathcal L$.     Similarly, it is clear that each of the solutions $u\in\mathcal L$ is
primitive.

In the following, we assume $a_3\not=0$.
It is clear that $p$ divides $a_4$; writing $a_4=p\wta_4$, Equation~\ref{rk4eq} becomes
$$a_1p^2+a_2(p^2+p)+a_3(p^2+p+1)=\wta_4(p^3+p^2+p+1).$$
Thus $p$ divides $a_3-\wta_4$; writing $a_3-\wta_4=p\wta_3$,
gives $$a_1p+a_2(p+1)+\wta_3(p^2+p+1)=\wta_4p^2.$$
Hence $p$ divides $a_2+\wta_3$; writing $a_2+\wta_3=p\wta_2$,
gives $$a_1+\wta_2(p+1)+\wta_3p=\wta_4p.$$
Finally, writing $a_1+\wta_2=p\wta_1$,
gives $$\wta_1+\wta_2+\wta_3=\wta_4.$$
Note that $\wta_1$ is in fact the height of the solution. In this new basis,
the constraints $a_1\geq 0$, $a_2\ge 0$, $a_3>0$ and $a_4> 0$ become
$$p\wta_1\ge \wta_2, \hspace{1cm} p\wta_2\ge \wta_3, 
\hspace{1cm} p\wta_3> -\wta_4, \hspace{1cm}
\wta_4>0.$$

Substituting and simplifying gives $a_1+a_3=(p+1)(\wta_1+\wta_3)$.
Since we are assuming $a_3>0$,  we have $\wta_1+\wta_3>0$. Suppose $Q$ has height one, in other words
$\wta_1=1$. Then $\wta_3\ge 0$. Since $p\wta_1\ge \wta_2$ and $p\wta_2\ge \wta_3$, we conclude
$\wta_2\in\{0,\ldots,p\}$ and $\wta_3\in\{0,\ldots,p\wta_2\}$. The $(i,j)$ solution from the family $\Delta$
gives a solution with $\wta_2=p-i$ and $\wta_3=p^2-pi-j$. Thus $\Delta$ is precisely the set of primitive solutions with 
height one and $a_3>0$.   In particular, being of height 1, each solution in $\Delta$ is primitive.

Suppose $\wta_1>1$ and $a_3>0$. If $a_1\ge p$, then $Q-v_{13}$ is non-negative, so we may assume
$a_1<p$. Thus $p\wta_1-\wta_2\le p-1$, giving $\wta_2\ge p\wta_1-p+1\ge 2p-p+1=p+1$.
Hence $a_2\ge p^2+p-\wta_3$. 
If  $\wta_3 \leq p-1$ then $a_2 \geq p^2+ p - \wta_3 \geq p^2 + p -(p-1) = p^2+1$ and so 
$Q-v_2$ is non-negative.  Therefore we may assume that $\wta_3 \geq p$.
Let $u$ denote the solution in $\Delta$ given by $i=a_1$ and $j=p^2-pa_1$.
Then $u=(a_1,p^2-pa_1,p+1-a_1,p^2+p-pa_1)$. Since $a_3=(p+1)(\wta_1+\wta_3)-a_1>(p+1)-a_1$,
we have $Q-u$ is non-negative unless $a_2<p^2-pa_1$. Thus we may assume $p^2-1\ge pa_1+a_2$.
Thus $p^2-1\ge p^2\wta_1-\wta_3\ge 2p^2-\wta_3$, giving $\wta_3\ge p^2+1$.
Hence $\wta_4=\wta_1+\wta_2+\wta_3\ge 2+(p+1)+(p^2+1)=p^2+p+4$, giving 
$a_3=\wta_4+p\wta_3\ge (p^2+p+4)+p(p^2+1)=p^3+p^2+2p+5>p^3+p^2+p+1$.
Therefore $Q-v_3$ is non-negative, giving the required contradiction.
\end{proof}

\begin{theorem}\label{rk4p^2}
Suppose $E$ is a subgroup of rank $4$. Then $E$ contains a dilation of $\F_{p^2}$ if and only if
$$\left(d_{1,4}^{p^2}d_{2,4}^pd_{3,4}-d_{4,4}^pd_{1,4}^{p^2+1}-d_{3,4}^{p^2+1}\right)(E)=0.$$
\end{theorem}
\begin{proof}
The argument is similar to the proof of Theorem~\ref{codim1}. The subgroup
$E={\rm Span}_{\F_P}\{c_1,c_2,c_3,c_4\}$ contains a dilation of $\F_{p^2}$ if and only if
there is a $g\in GL_4(\F_p)$ such that
$d_{1,2}(g(c_1,c_2,c_3,c_4))=0$. The stabiliser of $d_{1,2}$ 
has order $|GL_2(\F_p)|^2p^4$. Therefore the $GL_4(\F_p)$-orbit product of $d_{1,2}$ has degree
$${\rm deg}(d_{1,2})(p^2+1)(p^3-1)/(p-1)=p(p^2+1)(p^3-1).$$ The polynomial
$$H:=d_{1,4}^{p^2}d_{2,4}^pd_{3,4}-d_{4,4}^pd_{1,4}^{p^2+1}-d_{3,4}^{p^2+1}$$
is a non-zero homogeneous $GL_4(\F_p)$-invariant of degree $p(p^2+1)(p^3-1)$. We will show that $d_{1,2}$ divides $H$ and, 
therefore, $H$ is, up to multiplication by a non-zero scalar, the $GL_4(\F_p)$-orbit product of $d_{1,2}$.
Hence $H(E)=0$ if and only if $E$ contains a dilation of $\F_{p^2}$.

We first show that  $d_{1,2}$ divides $d_{1,3}^pd_{2,3}-d_{3,3}^p$  and then
work modulo the ideal $\mathfrak r:=\langle d_{1,3}^pd_{2,3}-d_{3,3}^p\rangle$.
The argument is essentially the initial case in the induction argument in the proof of Theorem~\ref{codim1}.
For convenience , we use $\widetilde{h}$ to denote $D_2^{p-1}(x_3)$. Thus, using the formula
$d_{i,3}=d_{i,2}^p-d_{i-1,2}\widetilde{h}$, we have 
$d_{1,3}=d^p_{1,2}-\widetilde{h}$,  $d_{2,3}=d^p_{2,2}-d_{1,2}\widetilde{h}$ and $d_{3,3}=-d_{2,2}\widetilde{h}$.
Therefore
 $$d_{1,3}^pd_{2,3}-d_{3,3}^p=(d^p_{1,2}-\widetilde{h})^p(d_{2,2}^p-d_{1,2}\widetilde{h})+d_{2,2}^p\widetilde{h}^p
 \equiv_{\langle d_{1,2}\rangle} 0.$$

We set $h=D_3^{p-1}(x_4)$ and use the
equations $d_{i,4}=d_{i,3}^p-d_{i-1,3}h$. Substituting and simplifying gives
$d_{2,4}^pd_{3,4}-d_{4,4}^pd_{1,4}\equiv_{\mathfrak r}d_{2,3}^{p^2}d_{3,4}$ and
$d_{1,4}d_{2,3}-d_{3,4}\equiv_{\mathfrak r}0$. Thus
$$H=d_{1,4}^{p^2}\left(d_{2,4}^pd_{3,4}-d_{4,4}^pd_{1,4}\right)-d_{3,4}^{p^2+1}\equiv_{\mathfrak r} 
d_{3,4}\left(d_{1,4}d_{2,3}-d_{3,4}\right)^{p^2}\equiv_{\mathfrak r} 0$$
as required.
\end{proof}

\begin{theorem} Suppose $E$ is a subgroup of rank $4$. Then
\begin{itemize}
\item[(i)] $\lambda(E)=(4)$ if and only if $d_{i,4}(E)=0$ for $i\in\{1,2,3\}$;
\item[(ii)] $\lambda(E)=(2,2)$ if and only if $d_{1,4}(E)=d_{3,4}(E)=0$ and $d_{2,4}(E)\not= 0$;
\item[(iii)] $\lambda(E)=(3,1)$ if and only if $d_{2,4}(E)=0$ and $v_{13}(E)=1$;
\item[(iv)] $\lambda(E)=(2,1,1)$ if and only if $d_{1,4}(E)\not=0$, $d_{2,4}(E)\not=0$ and\\ 
$\left(\left(v_{12}-v_{13}\right)^pv_{13}-v_1\right)(E)=0$.
\end{itemize}
\end{theorem}
\begin{proof}
Parts (i) and (ii) follow from Theorem~\ref{field-thm}.
Part (iii) follows from Theorem~\ref{codim1}. For part (iv), if $\lambda(E)=(2,1,1)$
then $E$ contains a dilation of $\F_{p^2}$ but does not contain a dilation of $\F_{p^3}$ and is not
an $\F_{p^2}$-subspace. Therefore, using   Theorem~\ref{rk4p^2}, $\lambda(E)=(2,1,1)$ if
and only if $H(E)=0$ but $d_{2,4}(E)\not=0$ and $d_{1,4}(E)\not=0$. Multiplying $H$ by 
$d_{1,4}^{p^3+p}/d_{4,4}^{p^3+1}$ gives  
$\left(v_{12}-v_{13}\right)^pv_{13}-v_1$.
\end{proof}

\begin{remark} It is not clear that the partition is the best way to understand the structure of a subgroup.
Consider $V=\F_{p^3}+\F_{p^2}\subset \overline{\F_p}$. This is a rank $4$ subgroup with partition $(3,1)$.
The equivalence class of the subgroup is determined by $v_{12}(V)=0$, $v_{13}(V)=1$ and $v_1(V)=-1$.
\end{remark}

\section{Invariants for Subgroups of Rank 5} \label{rnk5_sec}
\newcommand{\xx}[1]{{\widetilde{a}_#1}}
\newcommand{\xp}[1]{{\overline{x}'_#1}}
\newcommand{\xpp}[1]{{\overline{x}''_#1}}

In this section we will construct a generating set for $\F[\F^6]^{GL_5(\F_p)\times \F^\times}$.
To do this, we identify the primitive non-negative integer solutions to 
\begin{equation}\label{rk5eqn}
a_1p^4+a_2p^3\left(\frac{p^2-1}{p-1}\right)
+a_3p^2\left(\frac{p^3-1}{p-1}\right)+a_4p\left(\frac{p^4-1}{p-1}\right)=a_5\left(\frac{p^5-1}{p-1}\right).
\end{equation}
Since we are interested in non-zero solutions, we assume $a_5>0$. As in the proof of Theorem~\ref{rk4gen},
it will be useful to introduce an alternative basis:
$a_5=p\xx5$, $a_4 = p\xx4+\xx5$, $a_3=p\xx3 - \xx4$, $a_2 = p\xx2 - \xx3$ and $a_1 = p\xx1 - \xx2$.
In these new variables, the above equation becomes
$\xx1+\xx2+\xx3+\xx4=\xx5$,
the height of the solution, $a_1+a_2+a_3+a_4-a_5$, becomes $\xx1$ and the constraints $a_5>0$ and $a_i\geq 0$ 
for $i\in\{1,2,3,4\}$ become
\[ p\xx1\geq \xx2, \hspace{3mm}  p\xx2\geq \xx3,  \hspace{3mm}  p\xx3\geq \xx4,  \hspace{3mm} p\xx4\geq -\xx5,
\hspace{3mm}\xx5>0.\]
We will write $(a_1,a_2,a_3,a_4,a_5)=[\xx1,\xx2,\xx3,\xx4,\xx5]$.
It follows from Theorem~\ref{prime_sep} that orbits of finite subgroups of rank $5$ are separated by
$$\{v_1,v_2,v_3,v_4,v_{12},v_{13},v_{14},v_{23},v_{24},v_{34}\}.$$ We will extend this set to a generating set.
As in Section 6, we will identify an invariant monomial with its exponent sequence. To clarify our explanation of the generating set we will describe the exponent sequences as lattice points within various polyhedra.

We denote by $\mathcal T_1$ the set of lattice points in the tetrahedron with vertices
	\begin{align*}
	        v_4 =& (0,0,0,p^4+p^3+p^2+p+1,p^4+p^3+p^2+p)\\
	         &= \quad [1,p,p^2,p^3,p^3+p^2+p+1],\\
	        v_{14} =& (p,0,0,1,p) = [1,0,0,0,1],\\
		w_{24}:=&(0,p^2,0,p+1,p^2+p)=[1,p,0,0,p+1],\\
		w_{34}:=&(0,0,p^3,p^2+p+1,p^3+p^2+p)=[1,p,p^2,0,p^2+p+1].
	\end{align*}

Similarly $\mathcal T_2$ will denote 
the set of lattice points in the tetrahedron with vertices $v_{14}$, 
\begin{align*}
	        v_{23} =& (0,p^2,1,0,p^2)=[1,p,0,-1,p],\\
	        v_{34} =& (0,0,p^3+p,1,p^3+p) =  [1,p,p^2,-p,p^2+1],\\
	        w_{43}:=&(0,0,p^3+1,p^2,p^3+p^2)=[1,p,p^2,-1,p^2+p].
	\end{align*}

  Let $\Delta_1$ denote the set of lattice points in the triangle with vertices $v_{23}$, 
\begin{align*}
	        v_{32} =& (0,1,p^3,0,p^3)=[1,p,p^2-1,-p,p^2],\\
	        v_{312}:=&(p-1,1,p^2-p+1,0,p^2)=[1,1,p-1,-1,p].
	\end{align*}

Let $\Delta_2$ denote the set of 
lattice points in the triangle with vertices $v_{312}$, $v_{23}$ and
 $$v_{123}:=(p^2-1,p,1,0,p^2)=[p,1,0,-1,p].$$ 

Let $\Delta_3$ denote the set of lattice points in the triangle with vertices 
      \begin{align*}
                v_2 =& (0,p^4+p^3+p^2+p+1,0,0,p^4+p^3)\\
                &\quad=[p^2+p+1,p^3+p^2+p,-p-1,-p^2-p,p^3+p^2],\\
                w_{21}&:=(p,p^3+1,0,0,p^3)=[p+1,p^2,-1,-p,p^2],\\
                w_{24}&:=(0,p^3+p^2+1,0,p,p^3+p^2)=[p+1,p^2+p,-1,-p,p^2+p].		
	\end{align*}

Let $\mathcal L_1$ denote the set of lattice points on the line segment  joining 
\begin{align*}
                v_3 =& (0,0,p^4+p^3+p^2+p+1,0,p^4+p^3+p^2)\\
                &\quad=[p+1,p^2+p,p^3+p^2,-p^2-p-1,p^3+p^2+p],\\
                v_{13}&=(p^2+p,0,1,0,p^2)=[p+1,0,0,-1,p].
	\end{align*}

Let $\mathcal L_2$ denote the set of  lattice points on the line segment  joining $w_{21}$ to 
\begin{align*}
                v_{12} =& (p^3+p^2+p,1,0,0,p^3)=[p^2+p+1,0,-1,-p,p^2].
	\end{align*}

Then the generating set is the set
$$\mathcal S:=\left(\bigcup_{i=1}^2\mathcal T_i\right) \cup \left(\bigcup_{i=1}^3 \Delta_i\right)
     \cup \left(\bigcup_{i=1}^2\mathcal L_i\right) \cup \{v_1\}.$$
Note that 
\begin{align*}
                v_1 =& (p^4+p^3+p^2+p+1,0,0,0,p^4)=[p^3+p^2+p+1,-1,-p,-p^2,p^3].
	\end{align*}

\begin{theorem}\label{rk5gen} $\F[\F^6]^{\GL_5(\F_p)\times \F^\times}$ is minimally generated as an $\F$-algebra by $d_{5,5}t$ and
$$\{d_{1,4}^{a_1} d_{2,4}^{a_2} d_{3,4}^{a_3}d_{4,4}^{a_4}t^{a_5}\mid (a_1,a_2,a_3,a_4,a_5)\in \mathcal S \}.$$
\end{theorem}

The remainder of this section is dedicated to the proof of Theorem~\ref{rk5gen}

\begin{lemma}\label{tet1-lemma} The elements in $\mathcal T_1$ are precisely the solutions with height one and $\xx4\ge 0$.
\end{lemma}
\begin{proof}
Note that if $\xx4\ge 0$ then $a_4>0$ since $a_4 = p \xx4 + \xx5$.   For $i\in\{0,\ldots,p\}$, $j\in\{0,\ldots,p^2-ip\}$ and
$k\in\{0,\ldots,p^3-ip^2-jp\}$, define
$$u_{ijk}:=\frac{i}{p}v_{14}+\frac{j}{p^2}w_{24}+\frac{k}{p^3}w_{34}+\left(1-\frac{i}{p}-\frac{j}{p^2}-\frac{k}{p^3}\right)v_4.$$
The points $u_{ijk}$ are precisely the lattice points lying in the tetrahedron with vertices $v_{14}, w_{24}, w_{34}, v_4$.
Define $c_{ijk}:= p^4+(1-i)p^3+(1-i-j)p^2+(1-i-j-k)p$. Simplifying gives
$$u_{ijk}=\left(i,j,k,c_{ijk}+(1-i-j-k),c_{ijk}\right).$$
Write $\widetilde{i}:=p-i$,  $\widetilde{j}:=p^2-ip-j$, $\widetilde{k}:=p^3-p^2i-pj-k$. Then
$$u_{ijk}=[1,\widetilde{i},\widetilde{j},\widetilde{k},1+\widetilde{i}+\widetilde{j}+\widetilde{k}]$$
with $p\ge\widetilde{i}\ge 0$, $p\widetilde{i}\ge \widetilde{j}$ and $p\widetilde{j}\ge \widetilde{k}\ge 0$. 
From this it is clear that $\mathcal T_1$ consists of the solutions
with height one and $\xx4\ge 0$.
\end{proof}

\begin{lemma}\label{tet2-lemma} All solutions with height one, $a_4> 0$ and $\xx4<0$ lie in $\mathcal T_2$.
\end{lemma}
\begin{proof}
Substituting and simplifying gives $a_1+a_4=(p+1)(\xx1+\xx4)+\xx3$. Since we are assuming $a_4>0$, $\xx4<0$,
and $\xx1=1$, we conclude that $\xx3>0$. Using the basic constraints and the assumption $\xx1=1$, gives 
$p\ge \xx2>0$ and $p^2\ge p\xx2\ge \xx3$.
Since $p\xx4=a_4-\xx5\ge 1-\xx5=1-(\xx1+\xx2+\xx3+\xx4)=-\xx2-\xx3-\xx4$, we have
$\xx2+\xx3\ge-(p+1)\xx4$. Thus $p(p+1)\ge -\xx4(p+1)$ and $-p\le\xx4\le -1$.
Define
\begin{eqnarray*}
t_{ijk}:=\left(\frac{i}{p}-\frac{j}{p^2}\right)v_{23}&-&\left(\frac{k}{p-1}+\frac{i}{p(p-1)}\right)v_{34}\\
&+&\left(\frac{k}{p-1}+\frac{i}{p(p-1)}+\frac{j}{p^2}\right)w_{43}
+\left(1-\frac{i}{p}\right)v_{14}
\end{eqnarray*}
for $p\ge i\ge 0$, $pi\ge j\ge 0$ and $-1\ge k\ge -p$.
 Note that, as long as $i+j+(p+1)k\ge 0$, the coefficients are non-negative and the $t_{ijk}$
are precisely the lattice points lying in the tetrahedron with vertices $v_{23}, v_{34}, w_{43}, v_{14}$.
Thus each $t_{ijk}\in\mathcal T_2$.
Substituting and simplifying gives
\begin{eqnarray*}
t_{ijk}&=&\left(p-i,pi-j,pj-k,1+i+j+k(p+1),p(1+i+j+k)\right)\\
&=&[1,i,j,k,1+i+j+k].
\end{eqnarray*}
Thus every solution with height one, $a_4>0$ and $\xx4<0$ lies in $\mathcal T_2$.
\end{proof}

\begin{lemma}\label{delta1-lemma} The solutions with height one and $a_4=0$
 lie in $\Delta_1$.
\end{lemma}

\begin{proof}
Arguing as in the proof of Lemma~\ref{tet2-lemma}, we have $p^2\ge p\xx2\ge \xx3 \ge 0$.
Put $j := a_1=p-\xx2$ and $k:=a_2$.   
Then $k := p \xx2 - \xx3 \leq p \xx2 = p^2 - pj \leq p^2$.
Since $-p\xx4=\xx5-a_4=\xx5=1+\xx2+\xx3+\xx4$, we have $-(p+1)\xx4=\xx2+\xx3+1$.
Thus $-(p+1)\xx4 = \xx2 + (p\xx2-k)+1=(p+1)\xx2+1-k$ and thus $k\equiv 1 \pmod{p+1}$. 
Write $k = i(p+1) + 1$ where $0 \le i \le p-1$.
Define
$$t_{ij}:=\frac{i}{p-1}v_{23}+\frac{j}{p-1}v_{312}+\frac{p-1-i-j}{p-1}v_{32}$$
for $i,j\in\{0,\ldots,p-1\}$ with $i+j\le p-1$.
Substituting and simplifying gives
\begin{eqnarray*}
t_{ij}&=&\left(j,i(p+1)+1,p^3-j(p^2+1)-i(p^2+p+1),0,p^2(p-i-j)\right)\\
&=&[1,p-j,p^2-1-jp-i(p+1),-(p-i-j),p(p-i-j)].
\end{eqnarray*}
Thus the $t_{ij}$ are precisely the lattice points lying in the triangle with vertices $v_{23}, v_{312}, v_{32}$.
Thus each $t_{ij} \in \Delta_1$.
\end{proof}

\begin{lemma}\label{wta4nn-lemma} There are no primitive solutions with $\xx1 >1$ and $\xx4 \ge 0$.
\end{lemma}
\begin{proof}
Suppose, by way of contradiction, that $Q=(a_1,a_2,a_3,a_4,a_5)$ is a primitive solution with $\xx1>1$
and $\xx4\ge 0$. Suppose $\xx4\ge p^3$. Then, using the basic constraints, $\xx3\ge p^2$ and $\xx2\ge p$.
Thus $Q-v_4$ is non-negative. Thus we may assume $\xx3<p^3$. Similarly, using $w_{34}$, we may assume 
$\xx3<p^2$ and using $w_{24}$ we may assume $\xx2<p$. Consider the primitive solutions $u_{ijk}$ defined 
in the proof of Lemma~\ref{tet1-lemma}. Taking $i=p-\xx2$, $j=p^2-ip-\xx3$ and $k=p^3-p^2i-pj-\xx4$,
gives a primitive solution with $Q-u_{ijk}$ non-negative.
\end{proof}

  The solutions in $\mathcal T_1 \cup \mathcal T_2 \cup \Delta_1$ all have height 1 and so are all primitive.  
  In the next four remarks we describe the remaining elements of $S$.  In each case it is not hard to see that 
  each of the elements corresponds to a primitive solution.

\begin{remark}\label{delta2-remark} 
The solutions in $\Delta_2$
satisfy $\xx5=p$, $\xx4=-1$ and $a_4=0$. Define
$$s_{ij}:=\frac{i-1}{p-1}v_{23}+\frac{j-1}{p-1}v_{123}+\frac{p+1-(i+j)}{p-1}v_{312}$$
for $i,j\in\{1,\ldots,p\}$ and $i+j\le p+1$. Then
\begin{eqnarray*}
s_{ij}&=&
\left(pi-j,(j-1)(p+1)+i,p\left(p+1-(i+j)\right)+1,0,p^2\right)\\
&=&[i,j,p+1-(i+j),-1,p]
\end{eqnarray*}
lies in $\Delta_2$.
\end{remark}

\begin{remark}\label{delta3-remark} 
The solutions in $\Delta_3$ satisfy $a_3=0$, $a_4 \leq p$ and $\xx4 \leq -p$.
Define $$u_{ij}:=\frac{i}{p}w_{21}+\frac{j}{p}w_{24}+\frac{p-i-j}{p}v_2$$ for
$i,j\in\{0,\ldots,p\}$ with $i+j\le p$. Substituting and simplifying gives
\begin{eqnarray*}
u_{ij}&=&(i,p^4+p^2+(p^3+1)(1-(i+j))+p(1-i),0,j,p^4+p^3(1-(i+j))+pj)\\ 
  &=&[p^2+p+1,p^3+p^2+p,-(p+1),-(p^2+p),p^3+p^2]\\
  &+&(i+j)[-p,-p^2,1,p,-p^2]+[0,-i,0,0,j].
\end{eqnarray*}
Hence $u_{ij}\in\Delta_3$.
\end{remark}

\begin{remark}\label{l1rmk} Since 
both $v_3$ and $v_{13}$ satisfy
$a_2=0$, $a_4=0$ and $\xx4<0$
we see that the all solutions in $\mathcal L_1$ also satisfy
$a_2=0$, $a_4=0$ and $\xx4<0$. The lattice points on $\mathcal L_1$ are precisely the points
$u_i:=\left(iv_{13} +(p^2+p-i)v_3\right)/(p^2+p)$. 
Simplifying gives
\begin{eqnarray*}
u_{i}&=&(i,0,(p^2+p-i)(p^2+1)+1,0,(p^2+p+1-i)p^2)\\ 
  &=&[p+1,p^2+p-i,p^3+p^2-ip,-p^2-p-1+i,p^3+p^2+p-ip]
\end{eqnarray*}
%
for $i\in\{0,\ldots, p^2+p\}$.  
\end{remark}

\begin{remark}\label{l2rmk} Since 
both $v_{12}$ and $w_{21}$ satisfy $a_3=0$, $a_4=0$, $\xx4 =-p$ and $\xx5 =p^2$,
we see that all the solutions in $\mathcal L_2$ also satisfy $a_3=0$, $a_4=0$, $\xx4 =-p$ and $\xx5 =p^2$.
These solutions are precisely the lattice points given by
$$(p^3+p^2+p-j(p+1),1+jp,0,0,p^3)= [p^2+p+1-j,j,-1,-p, p^2]$$
for $j\in\{0,\ldots,p^2\}$.
\end{remark}

It remains to prove that $\F[\F^6]^{\GL_5(\F_p)\times \F^\times}$  
is generated by $d_{5,5}t$ and the monomial invariants whose exponent sequences lie in $\mathcal S$.
Suppose, by way of contradiction, that $Q=(a_1,a_2,a_3,a_4,a_5)$ is a primitive solution with $Q\not\in\mathcal S$.
Using Lemmas \ref{tet1-lemma},  \ref{tet2-lemma}, \ref{delta1-lemma} and \ref{wta4nn-lemma}, we may assume
$\xx1 >1$ and $\xx4 <0$.

We are left with four cases: 
\begin{itemize}
\item[(i)] $a_3=0$ and $a_4=0$;
\item[(ii)] $a_3=0$ and $a_4>0$;
\item[(iii)] $a_3>0$ and $a_4=0$;
\item[(iv)] $a_3>0$ and $a_4>0$.
\end{itemize}

\subsection*{Case (i)}
Suppose $a_3=0$, $a_4=0$, $\xx4<0$ and $\xx1>1$. If $a_2=0$, then $Q$ is a multiple of $v_1$, so we may also assume
$a_2>0$. Similarly, using $v_2$, we can assume $a_1>0$.
Since $a_3=a_4=0$, we have $a_5=p\xx5=-p^2\xx4=-p^3\xx3$. Thus $\xx3<0$ and Equation~\ref{rk5eqn}
becomes $a_1p+a_2(p+1)=-\xx3(p^4+p^3+p^2+p+1)$. 
Using $w_{21}=(p,p^3+1,0,0,p^3)$, we can assume that either $a_1<p$ or $a_2\le p^3$.

Suppose $a_1<p$. Write $a_1=p-j$ (so $j\in\{1,\ldots,p-1\}$).
Then $a_2(p+1)=(j-p)p-\xx3(p^4+p^3+p^2+p+1)$. 
Note that $j+1\equiv -\xx3 \pmod{p+1}$. Write $-\xx3=j+1+k(p+1)$ (so $k\ge 0$).
Thus $$a_2(p+1)\ge (j-p)p+(j+1)(p^4+p^3+p^2+p+1)=(p+1)((j+1)(p^3+p)+j-p+1).$$
Hence $a_2\ge (j+1)(p^3+p)+j-p+1=j(p^3+p+1)+p^3+1$.
Within $\Delta_3$, there are solutions on the line segment joining $v_2$ to $w_{21}$ of the form
$u_j=(p-j,p^3+1+j(p^3+p+1),0,0,p^3+jp^3))$ for $j\in\{0,\ldots,p\}$.
Therefore $Q-u_j$ is non-negative if $a_1 < p$.  

Suppose then that $p^3\ge a_2$. Define $k\in\{1,\ldots,p\}$ so that
$k\equiv a_2 \equiv -\xx3 \pmod p$. Note that $a_2-k=jp$ for $j\in\{0,\ldots,p^2-1\}$
Then $a_1p \ge -(p+1)a_2+k(p^4+p^3+p^2+p+1)=-pa_2 +k(p^4+p^3+p^2+p)-jp$.
Hence 
$$a_1\ge k(p^3+p^2+p+1)-a_2-j=k(p^3+p^2+p)-j(p+1).$$ The family $\mathcal L_2$ includes a solution
$u_j=(p^3+p^2+p-j(p+1),1+jp,0,0,p^3)$
for $j\in\{0,\ldots,p^2\}$. Therefore $Q-u_j$ is non-negative, giving the required contradiction.

\subsection*{Case (ii)}
Suppose $a_3=0$, $a_4>0$, $\xx4<0$ and $\xx1>1$. Since $a_3=0$, 
we have $\xx4=p\xx3$. 
Thus $\xx3<0$ and $p$ divides $\xx4$.
As in the proof of Lemma~\ref{tet2-lemma}, substituting and simplifying gives
$a_1+a_4=(p+1)(\xx1+\xx4)+\xx3$. Since $a_4>0$ and $\xx3<0$, we have $\xx1+\xx4>0$.
Using $v_{14}=(p,0,0,1,p)$, we may assume $a_1<p$. Substituting and simplifying 
gives $a_2=p^2\xx1-pa_1-\xx3$. Since $\xx1>-\xx4=-p\xx3$, we have
$a_2\ge p^2(1-p\xx3)-\xx3-pa_1=(-\xx3)(p^3+1)+p^2-pa_1$.

Suppose $a_1+a_4\ge p$.
Let $v$ denote the solution $u_{ij}\in\Delta_3$, as  defined in Remark~\ref{delta3-remark},
with $i=a_1$ and $j=p-i$. 
Then
$v=(i,p^3+p^2+1-ip,0,p-i,p^3+p(p-i))$. Since $-\xx3\ge 1$, we have $a_2\ge p^3+p^2+1-ip$.
Thus $Q-v$ is non-negative.

Suppose $a_1+a_4<p$. Note that $a_1+a_4-\xx3$ is a positive multiple of $p+1$.
Thus $-\xx3\ge p+1-(a_1+a_4)$ and 
$a_2\ge (p^3+1)(p+1-(a_1+a_4))+p^2-pa_1=p^4+p^3+p^2+p+1-(p^3+1)(a_1+a_4)-pa_1$.
Thus $Q-v$ is non-negative where $v$ is the solution $u_{ij}\in\Delta_3$ 
with $i=a_1$ and $j=a_4$.

\subsection*{Case (iii)} 
Suppose $a_3>0$, $a_4=0$, $\xx4<0$ and $\xx1>1$.
Since $a_4=0$, we have $a_5=-p^2\xx4$. Using $v_{13}$, we may assume
$a_1<p^2+p$ and using $v_{23}$ we may assume $a_2<p^2$.
Note that $(p+1)(\xx2+\xx4)+\xx1=a_2+\xx1+\xx2+\xx3+(p+1)\xx4=a_2$.
Thus $(p+1)(\xx2+\xx4)<a_2<p^2$ and $\xx4+\xx2< p-1$.

First we further suppose $a_2=0$. Equation~\ref{rk5eqn} becomes
$a_1p^4+a_3p^2(p^2+p+1)=a_5(p^4+p^3+p^2+p+1)$. Thus $a_1+a_3\equiv a_5 \pmod{p+1}$.
Write $a_1+a_3-a_5=k(p+1)$ with $k>0$. Then the height of the solution is
$\xx1=a_1+a_2+a_3+a_4-a_5=k(p+1)$. Hence $\xx2=p\xx1-a_1=kp(p+1)-a_1$.
Using $a_2=0$ gives $\xx3=p\xx2=kp^2(p+1)-a_1p$. We have two expressions for $a_3$:
$$a_3=k(p+1)-a_1+a_5=k(p+1)-a_1-p^2\xx4$$
and
$$a_3=p\xx3-\xx4=p^2\xx2-\xx4=p^3k(p+1)-p^2a_1-\xx4.$$
Equating these expressions and simplifying gives
$a_1-\xx4=k(p^2+p+1)$. Using this to eliminate $\xx4$ from the second expression for $a_3$
gives $a_3=p^3k(p+1)-p^2a_1+k(p^2+p+1)-a_1$. Let $v$ denote the solution $u_i\in\mathcal L_1$,
as defined in Remark~\ref{l1rmk}, with $i=a_1$. Since $k>0$, we see that $Q-v$ is non-negative.
Thus we may assume $a_2>0$. Furthermore, using $v_{32}$, we may assume $a_3<p^3$.

Suppose $\xx3<0$. Then $p^2>a_2=p\xx2-\xx3>p\xx2$. Hence $\xx2<p$.
Since $1\le a_3=p\xx3-\xx4$, we have $-\xx4\ge p+1$.
Thus $\xx1=\xx5-(\xx2+\xx3+\xx4)=(-\xx4)(p+1)-\xx2-\xx3\ge (p+1)^2+1-\xx2> p^2+p+2$.
Hence $a_1=p\xx1-\xx2>p^3+p^2+p$, contradicting our assumption that $a_1<p^2+p$. 
Thus we must have $\xx3\ge 0$. Furthermore, since $p\xx2\ge\xx3$, we get $\xx2\ge 0$.

Note that $\xx1+\xx2+\xx3=\xx5-\xx4=(-\xx4)(p+1)$. Thus if $\xx4=-1$, we have
$\xx3=p+1-(\xx1+\xx2)\ge 0$ and $Q=s_{ij}$ for $i=\xx1$ and $j=\xx2$, with 
$s_{ij}\in\Delta_2$ as defined in Remark~\ref{delta2-remark}. Thus we must have 
$\xx4 \leq -2$. 

Suppose $a_1=0$. Then $\xx2=p\xx1\ge 2p$ and $a_2\ge 2p^2-\xx3$.
Thus $\xx3>p^2$ since $a_2 < p^2$.
Combining $\xx4 - \xx5 =(p+1)\xx4$ with $(p+1)\xx1 = a_1 + \xx1 + \xx2$  yields $a_1=(p+1)(\xx1+\xx4)+\xx3$.
Thus $\xx3$ is divisible by $p+1$. Since $\xx3>p^2$, we have $\xx3\ge p(p+1)$.
Hence $a_3=p\xx3-\xx4\ge p^3+p^2+2$,
contradicting the assumption that $a_3<p$. Thus we must have $a_1>0$.

Suppose $a_1<p$. Let $u$ denote the solution $s_{ij}\in \Delta_2$ with $i=1$ and $j=p-a_1$.
Then $u=(a_1,p^2-a_1(p+1),pa_1+1,0,p^2)$. Thus either $a_2<p^2-a_1(p+1)$ or $a_3\le pa_1$.
First suppose $a_2<p^2-a_1(p+1)$. Then $pa_1+a_2\le p^2-a_1-1\le p^2-1$, which implies
$p^2-1\ge p^2\xx1-\xx3\ge 2p^2-\xx3$. Thus $\xx3\ge p^2+1$ and $a_3>p^3+p+2$, contradicting
the assumption $a_3<p^3$. Suppose $a_3\le pa_1$. Then $p\xx3-\xx4\le p^2-p$, giving 
$p\xx3\le p^2-p+\xx4\le p^2-p-2$. Thus $\xx3\le p-2$. Hence $a_2=p\xx2-\xx3\ge p\xx2-p+2$.
Since $a_2<p^2$, this forces $\xx2\le p$. Thus $a_1=p\xx1-\xx2\ge p\xx1-p\ge p$, giving the required contradiction.
Thus we must have $a_1\ge p$.

If $a_1 \ge p^2-1$ then using the solution $v_{123}=(p^2-1,p,1,0,p^2)$ we see that $a_2 < p$.
Next using the solutions $(ip-1,i,p^2-ip+1,0,p^2)$ from $\Delta_2$
with $i=a_2$ we have $a_3 \leq p^2-a_2p$.
This implies that $p^2 \geq p a_2 + a_3 = p^2 \xx2 - \xx4$ and thus $\xx2 =0$.  
From $\xx3 \geq 0$ and $ a_2 = p\xx2 - \xx3$ this yields  
$a_2=0$, giving a contradiction.  
Thus we must have $a_1 \leq p^2-2$ and so $p \leq a_1 \leq p^2-2$.  

The solution $v_{312}=(p-1,1,p^2-p+1,0,p^2)$ shows that $a_3 \leq p^2-p$, from which we derive that $\xx3 \leq p-2$.
If $a_2 < p$ then $p-1 \geq a_2 = p\xx2 - \xx3 \geq p\xx2 -p+2$ which yields $\xx2 \leq 1$.
Hence $\xx1 = -\xx4(p+1) - \xx2 -\xx3 \geq 2(p+1) - 1 - (p-2) = p +3$ and thus $a_1 \geq p(p+3)-1 = p^2 + 3p -1$.
Since $a_1 \leq p^2-2$ we have a contradiction and therefore $a_2 \geq p$.

Again consider the solutions $(ip-1,i,p^2-ip+1,0,p^2)$ from $\Delta_2$. 
Choose $i$ so that
$ip -1 \leq a_1 \leq ip+p-2$.  Since $a_2\ge p> i$, we have 
$a_3 \leq p^2 -ip \leq p^2 - a_1 + p - 2$ which implies $a_1 + a_3 \leq p^2 + p -2$.
Thus $p\xx1 - \xx2 + p\xx3 - \xx4 \leq p^2+p-2$ and
 $p(\xx1 + \xx3) \leq p^2 + p -2 + \xx2 + \xx4 < p^2+2p-3$ since $\xx2+\xx4 < p-1$.
 This yields $\xx1 + \xx3 \leq p+1$.   
 Therefore $-\xx4(p+1) = \xx1 + \xx2 + \xx3 \leq (p+1) + (p-2-\xx4)$ which implies $-\xx4 p \leq 2p-1$ and thus
 $\xx4 \geq -1$.  However $\xx4 \leq -2$, giving the required contradiction.

\subsection*{Case (iv)} 
Suppose $a_3>0$, $a_4>0$, $\xx4<0$ and $\xx1>1$.
Using $v_{14}=(p,0,0,1,p)$, we may assume $a_1<p$.
Using $v_{23}=(0,p^2,1,0,p^2)$, we may assume $a_2<p^2$.
Thus $\xx2=p\xx1-a_1>p$. Write $a_2=\alpha(p+1)+\beta$
with $0\le \alpha<p$ and $0\le \beta \le p$.
Let $u$ denote the solution $t_{ijk}\in\mathcal T_2$, as defined in the proof of
Lemma~\ref{tet2-lemma}, with $i=p$, $k=\alpha-p$ and $j=p^2-\alpha(p+1)$.
Then $u=(0,\alpha(p+1), p^3+p-\alpha(p^2+p+1),1,p^3+p-\alpha p^2)$. 
Thus $Q-u$ is non-negative if $a_3\ge  p^3+p-\alpha(p^2+p+1)$. Hence we may assume
$a_3< p^3+p-\alpha(p^2+p+1)\le p^3+p-p\alpha(p+1)\le p^3+p-p(a_2-\beta)\le p^3+p^2+p-pa_2.$
Therefore $p^3+p^2+p>a_3+pa_2=p^2\xx2-\xx4> p^2\xx2$. From this we get $p+1\ge \xx2$.
Since we also have $\xx2>p$, we conclude that $\xx2=p+1$. 

Since $\xx1 > 1$ we have $a_1=p\xx1-\xx2 \geq 2p-(p+1)=p-1$.  But $a_1 <p$ and thus $a_1=p-1$.
Furthermore, $p^2>a_2=p\xx2-\xx3=p(p+1)-\xx3$ gives $\xx3 \geq p+1$. Thus $a_3=p\xx3-\xx4>p^2+p$.
Let $v$ denote the solution $t_{ijk}\in \mathcal T_2$ with $i=1$, $j=p$ and $k=-1$. Then
$v=(p-1,0,p^2+1,1,p^2+p)$ and $Q-v$ is non-negative, giving the required contradiction.  

\section{Partitioning Rank 5 Subgroups}\label{rnk5_par}

Applying Theorems \ref{field-thm} and \ref{codim1}, gives the following.
\begin{theorem} Suppose $E$ is a subgroup of rank $5$. Then
\begin{itemize}
\item[(i)] $\lambda(E)=(5)$ if and only if $d_{i,5}(E)=0$ for $i\in\{1,2,3,4\}$;
\item[(ii)] $\lambda(E)=(4,1)$ if and only if $d_{i,5}(E)=0$ for $i\in\{2,3\}$ and $v_{14}(E)=1$.
\end{itemize}
\end{theorem}

We do not have a complete characterisation of the partitions of the rank $5$ subgroups. However the following approach
has lead to a range of conjectures which have been verified for small primes using Magma \cite{magma}. 

Use $h$ to denote $D_4^{p-1}(x_5)$. Then $d_{i,5}=d_{i,4}^p-d_{i-1,4}h\in A:=\F[d_{1,4},\ldots,d_{4,4},h]$.
Thus $\F[d_{1,5},\ldots,d_{5,5}]$ is a subalgebra of $A$. Using Theorem~\ref{codim1}, to show that a rank $5$ subgroup 
$E={\rm Span}_{\F_p}\{c_1,\ldots,c_5\}$ contains a 
dilation of $\F_{p^3}$, it is sufficient to find $g\in GL_5(\F_p)$ such that $d_{2,4}g(c_1,\ldots,c_5)=0$
and $$\left(d_{1,4}^pd_{3,4}-d_{4,4}^p\right)g(c_1,\ldots,c_5)=0.$$
In other words, we want to consider the union of the varieties determined by 
$\langle d_{2,4}g, \left(d_{1,4}^pd_{3,4}-d_{4,4}^p\right)g\rangle$. 
This is equivalent to the variety determined by the intersection of these ideals.
Computing this intersection is equivalent to computing
$$\langle d_{2,4}, \left(d_{1,4}^pd_{3,4}-d_{4,4}^p\right)\rangle\cap\F[d_{1,5},\ldots,d_{5,5}].$$
This can be done computationally as an elimination of variables. Using this approach, the following has been verified for
primes less than or equal to $13$.

\begin{conjecture} A rank $5$ subgroup $E$ contains a dilation of $\F_{p^3}$ if and only if
$E$ is in the variety determined by $d_{4,5}^p-d_{3,5}d_{1,5}^p$, $d_{5,5}^pd_{1,5}-d_{4,5}d_{2,5}^p$ and
$d_{5,5}^{ip}d_{4,5}^{p-i}-d_{3,5}d_{2,5}^{ip}d_{1,5}^{p-i}$ for $i\in\{1,\ldots,p\}$.
\end{conjecture}

If $E$ is in the variety determined by $\langle d_{1,4},d_{3,4} \rangle\cap \F[d_{1,5},\ldots,d_{5,5}]$,
then $E$ contains a rank $4$ subgroup which is a vector space over $\F_{p^2}$.

\begin{conjecture}\label{subspace} 
 $\langle d_{1,4},d_{3,4} \rangle\cap \F[d_{1,5},\ldots,d_{5,5}]=\langle d_{5,5}^p-d_{4,5}d_{1,5}^p,d_{3,5}^p-d_{2,5}d_{1,5}^p\rangle$.
\end{conjecture}

Conjecture~\ref{subspace} has been verified for primes less than or equal to $17$.

If $E$ has partition $(4,1)$ then $E$ is in the variety determined by
$$\langle d_{5,5}^p-d_{4,5}d_{1,5}^p,d_{3,5},d_{2,5}\rangle \subset \langle d_{5,5}^p-d_{4,5}d_{1,5}^p,d_{3,5}^p-d_{2,5}d_{1,5}^p\rangle.$$
Thus we expect a subgroup with partition $(2,2,1)$ to satisfy $d_{2,5}(E)\not=0$.

Since $d_{5,5}(E)$ is non-zero for any rank 5 subgroup, we can replace  $d_{5,5}^p-d_{4,5}d_{1,5}^p$
with $v_{14}-1$. If $(v_{14}-1)(E)=0$, then $d_{1,5}(E)\not=0$ and so we can replace
$d_{3,5}^p-d_{2,5}d_{1,5}^p$ with $v_{13}^p-v_{12}$. Thus
$$\langle d_{5,5}^p-d_{4,5}d_{1,5}^p,d_{3,5}^p-d_{2,5}d_{1,5}^p\rangle=\langle v_{14}-1,v_{13}^p-v_{12}\rangle.$$

\begin{example} $\F_{p^6}$ is an extension of both $\F_{p^2}=\F_p(b)$ and $F_{p^3}=\F_p(a)$.
The subgroup $E:={\rm Span}_{\F_p}\{1,a,a^3,b,ba\}$ contains the subgroup $\{1,b,a,ab\}$, 
which is a vector space over $\F_{p^2}$.
However $E=\F_{p^3}\oplus (a+b)\F_{p^2}$ so the partition is  $(3,2)$ rather than $(2,2,1)$.
Note that, using Theorem~\ref{embedding-thm}, a rank 5 subgroup $E$ is a subgroup of a dilation of
$\F_{p^6}$ if and only if $d_{1,5}(E)\not=0$ and $v_1(E)=v_{12}(E)^p=v_{13}(E)^{p^2}=v_{14}(E)^{p^3}$.
\end{example}

Using Theorem~\ref{rk4p^2}, if $E$ is in the variety determined by 
$$\langle d_{1,4}^{p^2}d_{2,4}^pd_{3,4}-d_{4,4}^pd_{1,4}^{p^2+1}-d_{3,4}^{p^2+1}\rangle \cap \F[d_{1,5},\ldots,d_{5,5}]$$ then
$E$ contains a dilation of $\F_{p^2}$. For convenience, denote
$r_1:=d_{5,5}^p-d_{4,5}d_{1,5}^p$ and $r_2:=d_{3,5}^p-d_{2,5}d_{1,5}^p$.

\begin{conjecture}\label{rk5p^2}
$$\langle d_{1,4}^{p^2}d_{2,4}^pd_{3,4}-d_{4,4}^pd_{1,4}^{p^2+1}-d_{3,4}^{p^2+1}\rangle \cap \F[d_{1,5},\ldots,d_{5,5}]=
\langle r_1^{p^2+1}-r_2^{p^2}(d_{5,5}^pd_{2,5}-d_{4,5}d_{3,5}^p)\rangle.$$
\end{conjecture}

Conjecture~\ref{rk5p^2} has been verified for primes less than or equal to $23$.

Thus we expect a subgroup $E$ with partition $(2,1,1,1)$ to satisfy
$$\left(r_1^{p^2+1}-r_2^{p^2}\left(d_{5,5}^pd_{2,5}-d_{4,5}d_{3,5}^p\right)\right)(E)=0$$ with either
$r_1(E)\not=0$ or $r_2(E)\not=0$.

\section{Acknowledgements}
  We thank the anonymous referee for a very careful reading of an earlier draft and for a number of helpful comments.
 The computer algebra program Magma \cite{magma} was very helpful in our study of this problem.  

\ifx\undefined\bysame
\newcommand{\bysame}{\leavevmode\hbox to3em{\hrulefill}\,}
\fi

\end{document}